\documentclass[a4paper,12pt]{amsart}

\usepackage{amsfonts}
\usepackage{amscd, amssymb}
\usepackage{amsmath,amscd}

\usepackage[hmargin=2cm,vmargin=2cm]{geometry}

\usepackage[driverfallback=hypertex]{hyperref}
\usepackage{nameref,zref-xr}                    %to include reference to other papers
\zxrsetup{toltxlabel}
\zexternaldocument*[calmodulin-]{real_Moduli_20_11}

\usepackage{comment}
\usepackage{graphicx}
\usepackage[usenames,dvipsnames]{color}
\usepackage{bm}
\usepackage{enumitem}
\usepackage[all]{xy}   %For commutative diagrams
\setlist[enumerate,1]{label={(\alph*)}}
\setlist[enumerate,2]{label={(\roman*)}}
\input xy
\xyoption{all}
\usepackage{epstopdf}

\newtheorem{thm}{Theorem}[section]
\newtheorem{prop}[thm]{Proposition}

\newtheorem{lemma}[thm]{Lemma}
\newtheorem{cor}[thm]{Corollary}

\newtheorem{conj}[thm]{Conjecture}

\theoremstyle{definition}

\theoremstyle{remark}
\newtheorem{rmk}[thm]{Remark}

\renewcommand{\O}{\mathcal{O}}

\newcommand{\Z}{\mathbb Z}

\newcommand{\CL}{{\mathbb{L}}}

\newcommand{\cW}{\mathcal{W}}

\newcommand{\I}{\mathcal{I}}

\newcommand{\M}{\overline{\mathcal{M}}}

\DeclareMathOperator{\rk}{rk}
\DeclareMathOperator{\deg1}{deg}

\renewcommand{\d}{\partial}
\newcommand{\mbC}{\mathbb{C}}
\newcommand{\eps}{\varepsilon}
\DeclareMathOperator{\res}{res}
\newcommand{\hP}{\widehat{P}}

\newcommand{\<}{\left<}
\renewcommand{\>}{\right>}
\newcommand{\hL}{\widehat{L}}
\newcommand{\mbZ}{\mathbb{Z}}

\newcommand{\mcL}{\mathcal{L}}
\newcommand{\tphi}{\widetilde{\phi}}
\newcommand{\tGamma}{\widetilde{\Gamma}}
\newcommand{\oM}{\overline{\mathcal{M}}}
\newcommand{\ext}{\mathrm{ext}}
\newcommand{\mcC}{\mathcal{C}}

\newcommand{\mcS}{\mathcal{S}}
\newcommand{\ti}{\widetilde{i}}

\numberwithin{equation}{section}

\begin{document}

\title{Closed extended $r$-spin theory and the Gelfand--Dickey wave function}

\author{Alexandr Buryak}
\address{A.~Buryak:\newline School of Mathematics, University of Leeds, Leeds, LS2 9JT, United Kingdom, and \newline
Faculty of Mechanics and Mathematics, Lomonosov Moscow State University, Moscow, GSP-1, 119991, Russian Federation}
\email{buryaksh@gmail.com}

\author{Emily Clader}
\address{E.~Clader:\newline San Francisco State University, San Francisco, CA 94132-1722, USA}
\email{eclader@sfsu.edu}

\author{Ran J. Tessler}
\address{R.~J.~Tessler, Incumbent of the Lilian and George Lyttle Career Development Chair:\newline Department of Mathematics, Weizmann Institute of Science, POB 26, Rehovot 7610001, Israel}
\email{ran.tessler@weizmann.ac.il}

\thanks{The authors would like to thank Rahul Pandharipande and J\'er\'emy Gu\'er\'e for many fruitful discussions. The work of A. B. (Section 4) was supported by the grant no. 16-11-10260 of the Russian Science Foundation. E.~C. was partially supported by NSF DMS Grant 1810969. R.~T. was supported by a research grant from the Center for New Scientists of Weizmann Institute, and by Dr. Max R\"ossler, the Walter Haefner Foundation and the ETH Z\"urich Foundation.}

\begin{abstract}
We study a generalization of genus-zero $r$-spin theory in which exactly one insertion has a negative-one twist, which we refer to as the ``closed extended" theory, and which is closely related to the open $r$-spin theory of Riemann surfaces with boundary.  We prove that the generating function of genus-zero closed extended intersection numbers coincides with the genus-zero part of a special solution to the system of differential equations for the wave function of the $r$-th Gelfand--Dickey hierarchy.  This parallels an analogous result for the open $r$-spin generating function in the companion paper \cite{BCT} to this work.
\end{abstract}

\maketitle

%\tableofcontents

\section{Introduction}\label{sec:intro}

Witten's conjecture \cite{Witten2DGravity}, which was proposed in 1991 and soon proven by Kontsevich \cite{Kontsevich}, states that the generating function for the integrals of the cotangent line classes $\psi_1, \ldots, \psi_n \in H^2(\M_{g,n})$ on the moduli space of curves is governed by the Korteweg--de Vries (KdV) hierarchy.  At around the same time, Witten also proposed a generalization of his conjecture \cite{Witten93}, in which the moduli space of curves is enhanced to the moduli space $\M_{g,\{\alpha_1, \ldots, \alpha_n\}}^{1/r}$ of $r$-spin structures.  The latter is a natural compactification of the space of smooth marked curves $(C;z_1, \ldots, z_n)$ with a line bundle $S$ and an isomorphism
\[
S^{\otimes r} \cong \omega_C \left(-\sum_{i=1}^n \alpha_i[z_i]\right),
\]
where $\alpha_i \in \{0,1,\ldots, r-1\}$.  This space admits a virtual fundamental class $c_W$, which is referred to as {\it Witten's class} and is defined in genus zero by
\[c_W:= e((R^1\pi_*\mathcal{S})^{\vee}),\]
with $\pi: \mathcal{C} \rightarrow \M_{0,\{\alpha_1, \ldots, \alpha_n\}}^{1/r}$ the universal curve and $\mathcal{S}$ the universal $r$-spin structure; constructions of $c_W$ in higher genus have now been given by a number of authors \cite{PV,ChiodoWitten,Moc06,FJR,CLL}. Let $t^{\alpha}_d$ and~$\eps$ be formal variables, for $0\le \alpha \le r-1$ and $d\ge 0$. The $r$-spin Witten conjecture states that if
\[
F^{\frac{1}{r},c}(t^*_*, \eps): = \sum_{\substack{g \geq 0,\, n\geq 1\\2g-2+n>0}} \sum_{\substack{0 \leq \alpha_1, \ldots, \alpha_n \leq r-1\\ d_1, \ldots, d_n \geq 0}} \frac{\eps^{2g-2}}{n!}\left(r^{1-g}\int_{\M^{1/r}_{g,\{\alpha_1, \ldots, \alpha_n\}}} \hspace{-1cm} c_W \cap \psi_1^{d_1} \cdots \psi_n^{d_n} \right)t^{\alpha_1}_{d_1} \cdots t^{\alpha_n}_{d_n}
\]
is the generating function of $\psi$-intersection numbers against Witten's class, then $\exp(F^{\frac{1}{r},c})$ becomes, after a certain rescaling of the variables~$t^\alpha_d$, a tau-function of the $r$-th Gelfand--Dickey ($r$-GD) hierarchy in the standard normalization of flows~\cite{Dic03}. The superscript ``$c$", which stands for ``closed", is to distinguish this theory from the open theory discussed below.  This result was proven by Faber--Shadrin--Zvonkine \cite{FSZ10}.

If one allows that exactly one of the indices $\alpha_i$ is equal to $-1$ and the rest lie in the range $\{0,1,\ldots, r-1\}$, then the space $\M_{0,\{\alpha_1, \ldots, \alpha_n\}}^{1/r}$ is still defined and $R^1\pi_*\mathcal{S}$ is still a vector bundle in genus zero, so genus-zero $r$-spin intersection numbers can be defined exactly as above. This kind of intersection number was first considered in the paper~\cite{JKV2}. We refer to this theory as {\it closed extended $r$-spin theory}, and we define a genus-zero generating function by
\[
F_0^{\frac{1}{r},\ext}(t_*^*):=\sum_{n \geq 2} \sum_{\substack{0 \leq \alpha_1, \ldots, \alpha_n \leq r-1\\ d_1, \ldots, d_n \geq 0}} \frac{1}{n!}\left(r\int_{\M^{1/r}_{0,\{\alpha_1, \ldots, \alpha_n,-1\}}} \hspace{-1cm} c_W \cap \psi_1^{d_1} \cdots \psi_n^{d_n} \right)t^{\alpha_1}_{d_1} \cdots t^{\alpha_n}_{d_n}.
\]

In the current paper, we prove that $F_0^{\frac{1}{r}, \ext}$ coincides with the genus-zero part of a special solution of the system of differential equations for the wave function of the $r$-GD hierarchy. The proof of this result proceeds by verifying that the closed extended $r$-spin correlators satisfy certain topological recursion relations that allow the entire theory to be recovered from just two initial conditions that can be explicitly calculated.  We then verify that the genus-zero part of the special solution satisfies the same recursions and the same initial conditions.

The reason for our interest in the closed extended generating function arises from an intriguing connection to open $r$-spin theory, which is the generalization of $r$-spin theory to Riemann surfaces with boundary.  In \cite{BCT}, we construct a moduli space $\M_{0,k,l}^{1/r}$ of ``graded $r$-spin disks" that generalizes $\M_{0,n}^{1/r}$ to genus-zero curves $C$ equipped with an involution that realizes $C$ as two copies of a Riemann surface with boundary $\Sigma$, glued along their common boundary.  The moduli space~$\M_{0,k,l}^{1/r}$ itself has boundary and is not necessarily canonically oriented, so one must prescribe boundary conditions for sections of bundles and specify relative orientations in order to ensure that integration of their relative top Chern class is well-defined. After carrying out this technical work, we obtain in \cite{BCT} a definition of open $r$-spin correlators. When calculating recursions for open genus-zero correlators, the closed extended correlators appear naturally, and in fact, there is an intimate relation between the genus-zero sectors of the two theories that we call the open-closed correspondence; see Section~\ref{sec:o-c corr} below. The origin of this correspondence is at present mysterious.

\subsection{Plan of the paper}

In Section~\ref{sec:background}, we recall the relevant background information on closed (non-extended) $r$-spin theory.  Section~\ref{sec:geo} then generalizes the definitions to closed extended $r$-spin theory, proves the topological recursion relations, and calculates the two correlators that form the initial conditions for the potential $F_0^{\frac{1}{r},\ext}$.  We turn to a detailed treatment of the integrable hierarchy in Section~\ref{sec:alg}, which allows us to state the main result of the paper, Theorem~\ref{theorem:closed extended and wave}, and to prove it. Finally, in Section~\ref{sec:o-c corr}, we explain the correspondence between closed extended and open $r$-spin theory.

\section{Background on $r$-spin theory}\label{sec:background}

We begin by reviewing the relevant background on the moduli space of $r$-spin structures and its intersection theory, referring the reader to \cite{ChiodoStable,JKV}, among many other references, for more details.

Throughout what follows, fix an integer $r \geq 2$. An {\it $r$-spin structure} on a smooth marked curve $(C;z_1, \ldots, z_n)$ of genus $g$ is a line bundle $L$ on $C$ together with an isomorphism
\begin{equation}
\label{eq:Lr}
L^{\otimes r} \cong \omega_{C,\log}:= \omega_C\left(\sum_{i=1}^n[z_i]\right).
\end{equation}
There is a smooth Deligne--Mumford stack $\mathcal{M}^{1/r}_{g,n}$ parameterizing such objects, equipped with a finite \'etale morphism to (indeed, a torsor structure over) the moduli space $\mathcal{M}_{g,n}$ of smooth curves.  Some care must be taken in the compactification in order to preserve these properties of the moduli space of $r$-spin structures, and there are several ways to do so, as summarized in~\cite[Section 2.2]{CZ}.  In our case, we compactify by allowing orbifold structure.

More precisely, following \cite{AV}, we define an {\it orbifold curve} as a one-dimensional Deligne--Mumford stack with a finite ordered collection of marked points and at worst nodal singularities such that
\begin{enumerate}
\item the only points with nontrivial isotropy are marked points and nodes;
\item all nodes are {\it balanced}---i.e., in the local picture $\{xy=0\}$ at a node, the action of the distinguished generator of the isotropy group $\Z_k$ is given by
\[(x,y) \mapsto (\zeta_k, \zeta_k^{-1}y),\]
where $\zeta_k$ is a primitive $k$th root of unity.
\end{enumerate}
An orbifold curve is said to be {\it $r$-stable}, following \cite{ChiodoStable}, if the coarse underlying marked curve is stable and the isotropy group is $\Z_r$ at every special point.

Let $(C;z_1, \ldots, z_n)$ be an $r$-stable curve.  An {\it $r$-spin structure} on $C$ is an orbifold line bundle~$L$ together with an isomorphism as in \eqref{eq:Lr}.  If $z_0 \in C$ is either a marked point or a branch of a node, then the {\it multiplicity} of $L$ at $z_0$ is defined as the integer $m \in \{0,1,\ldots, r-1\}$ such that, in local coordinates $(z,v)$ on the total space of $L$ near $z_0$, the action of the distinguished generator of the isotropy group $\Z_r$ at $z_0$ is given by
\[ (z,v) \mapsto (\zeta_r z, \zeta_r^mv).\]
A standard but crucial fact about the multiplicities is that they determine the relationship between $L$ and its pushforward $|L|$ to the coarse underlying curve.  Specifically, suppose that $C' \subset C$ is an irreducible component with special points $\{z_k\}$ at which the multiplicities of $L$ are $\{m_k\}$.  Then, if $\rho: C' \rightarrow |C'|$ is the natural map to the coarse underlying curve, we have
\[L|_{C'} = \rho^*\left(|L|\big|_{|C'|}\right) \otimes \O_{C'}\left(\sum_k \frac{m_k}{r}[z_k]\right).\]
Given that $\omega_{C',\log} = \rho^*\omega_{|C'|, \log}$, we find that $|L|\big|_{|C'|}$ satisfies the equation
\begin{equation}
\label{eq:mults}
\left(|L|\big|_{|C'|}\right)^{\otimes r} \cong \omega_{|C'|,\log} \left(-\sum_{k} m_k[z_k]\right).
\end{equation}
Using \eqref{eq:mults}, one can prove (see, for example, the appendix of \cite{CR}) that there is an equivalence of categories between $r$-spin structures as above and orbifold line bundles $L$ with an isomorphism
\begin{equation*}
L^{\otimes r} \cong \omega_{C} \left(-\sum_{i=1}^n \mu_i[z_i]\right)
\end{equation*}
for which $\mu_i \in \{-1,0,1,\ldots, r-2\}$ and the isotropy groups at all markings act trivially on the fiber of $L$.  Finally, replacing $L$ by $S:=L(-\sum_{\mu_i=-1} [z_i])$, we find that there is a further equivalence with the category of orbifold line bundles satisfying
\begin{equation}
\label{eq:twistedrspin}
S^{\otimes r} \cong \omega_{C} \left(-\sum_{i=1}^n \alpha_i[z_i]\right)
\end{equation}
with
\[\alpha_i \in \{0,1,\ldots, r-1\},\]
where again, the isotropy groups at all markings act trivially on the fiber of $S$.  We view $r$-spin structures as in \eqref{eq:twistedrspin} in what follows, and we refer to the integers $\alpha_i$ as {\it twists}.  When $\alpha_i=r-1$, we say that $z_i$ is a {\it Ramond} marked point, and otherwise, it is said to be {\it Neveu--Schwarz}.

There is a proper, smooth Deligne--Mumford stack $\M_{g,n}^{1/r}$ of dimension $3g-3+n$ parameterizing $r$-stable curves together with an orbifold line bundle $S$ satisfying \eqref{eq:twistedrspin}.  It is equipped with a decomposition into open and closed substacks
\[\M^{1/r}_{g,\{\alpha_1, \ldots, \alpha_n\}} \subset \M^{1/r}_{g,n}\]
on which $S$ has twist $\alpha_i$ at $z_i$ for each $i \in \{1, \ldots, n\}$.  Furthermore, it is equipped with a virtual fundamental class $c_W$ from which a beautiful intersection theory can be defined.  The construction of $c_W$, which we refer to as {\it Witten's class}, was first suggested by Witten in genus zero.  Specifically, it is straightforward to check that in genus zero, if $\pi: \mathcal{C} \rightarrow \M_{g,n}^{1/r}$ denotes the universal curve and $\mathcal{S}$ the universal line bundle, then $R^0\pi_*\mathcal{S} = 0$ and hence $R^1\pi_*\mathcal{S}$ is a vector bundle.  We define the {\it Witten bundle} as
\[\mathcal{W}:= (R^1\pi_*\mathcal{S})^{\vee} = R^0\pi_*(\mathcal{S}^{\vee} \otimes \omega_{\pi})\]
and set $c_W$ to be its top Chern class:
\[c_W := e(\mathcal{W}).\]
Via the Riemann--Roch formula, one can check that the restriction of $c_W$ to $\M^{1/r}_{0,\{\alpha_1, \ldots, \alpha_n\}}$ has complex codimension
\begin{equation*}
\frac{-(r-2) + \sum_{i=1}^n \alpha_i}{r},
\end{equation*}
which is a non-negative integer if and only if $\M^{1/r}_{0,\{\alpha_1, \ldots, \alpha_n\}}$ is nonempty.

It is interesting and highly non-trivial to find the appropriate generalization of this definition to higher genus.  Various constructions have now been given, by Polishchuk--Vaintrob \cite{PV}, Chiodo \cite{ChiodoWitten}, Mochizuki~\cite{Moc06}, Fan--Jarvis--Ruan \cite{FJR}, and Chang--Li--Li \cite{CLL}. The result, in any of these cases, is a class on $\M^{1/r}_{g,\{\alpha_1, \ldots, \alpha_n\}}$ of complex codimension
\begin{equation}
\label{eq:rankW}
e:=\frac{(g-1)(r-2) + \sum_{i=1}^n \alpha_i}{r}.
\end{equation}
All of these constructions have been shown to agree after pushforward to $\M_{g,n}$ \cite[Theorem 3]{PPZ}, at least when all insertions are Neveu--Schwarz (which, by the Ramond vanishing explained below, is all that is needed).

One obtains correlators by integrating Witten's class against $\psi$-classes on the moduli space.  Namely, for each $i\in \{1, \ldots, n\}$, let $\mathbb{L}_i$ be the cotangent line bundle to the coarse curve $|C|$ at the $i$th marked point.  Then, in genus zero, we define closed $r$-spin correlators by
\[\left\langle \prod_{i=1}^n \tau_{d_i}^{\alpha_i}\right\rangle_0^{\frac{1}{r}} := r\int_{\M^{1/r}_{0,\{\alpha_1, \ldots, \alpha_n\}}} e\left( \mathcal{W} \oplus \bigoplus_{i=1}^l \mathbb{L}_i^{\oplus d_i}\right),\]
which is nonzero only if the equation
\begin{equation}\label{eq:when_int_is_number_closed_extended}
e+\sum_{i=1}^n d_i = 3g-3+n
\end{equation}
is satisfied with $g=0$.

\begin{rmk}
Putting the coefficient $r$ in front of the integral in the definition of the correlators $\left\langle \prod_{i=1}^n \tau_{d_i}^{\alpha_i}\right\rangle_0^{\frac{1}{r}}$ and, more generally, putting the rescaling coefficient $r^{1-g}$ in the definition of the generating function $F^{\frac{1}{r},c}(t^*_*, \eps)$ is a matter of tradition. This allows to present some results in a slightly more compact form.
\end{rmk}

A crucial fact that we require about these correlators is that they satisfy {\it Ramond vanishing}: if $\alpha_i=r-1$ for some $i$, then the correlator is zero.  To prove this, suppose that $S$ satisfies \eqref{eq:twistedrspin} and that $\alpha_1 = r-1$.  Let $\mathcal{S}$ be the universal line bundle on the universal curve $\pi: \mathcal{C} \rightarrow \M^{1/r}_{0,\{\alpha_1, \ldots \alpha_n\}}$, let $\Delta_1 \subset \mathcal{C}$ be the divisor corresponding to the first marked point, and let $\widetilde{\mathcal{S}} := \mathcal{S}\left(\Delta_1\right)$.  Then there is an exact sequence
\begin{equation}\label{eq:vanishingles}
0 \rightarrow R^0\pi_*\mathcal{S} \rightarrow R^0\pi_*\widetilde{\mathcal{S}} \rightarrow \sigma_1^*\widetilde{\mathcal{S}} \rightarrow R^1\pi_*\mathcal{S} \rightarrow R^1\pi_*\widetilde{\mathcal{S}}\rightarrow 0,
\end{equation}
where $\sigma_1$ is the section corresponding to the first marked point.  We have $R^0\pi_*\widetilde{\mathcal{S}} = 0$ and
\[(\sigma_1^*\widetilde{\mathcal{S}})^{\otimes r} \cong \sigma_1^*\omega_{\pi, \log} \cong \O_{\M_{0,\{\alpha_1 ,\ldots, \alpha_n\}}^{1/r}},\]
which implies that $e(\sigma_1^*\widetilde{\mathcal{S}}_1) = 0$.  By multiplicativity of the Euler class in \eqref{eq:vanishingles}, this implies that $c_W = 0$.

\section{Closed extended theory: geometry}\label{sec:geo}

Although we have thus far only defined the Witten class $c_W$ under the assumption that all twists lie in the range $\{0,1,\ldots, r-1\}$, there exists a smooth Deligne--Mumford moduli stack~$\M_{g,\{\alpha_1, \ldots, \alpha_n\}}^{1/r}$ parameterizing $r$-stable curves with a line bundle $S$ satisfying \eqref{eq:twistedrspin} for {\it any} tuple of integers $\{\alpha_1, \ldots, \alpha_n\}$.  This observation was made by Jarvis--Kimura--Vaintrob \cite{JKV2}, who studied precisely how the virtual class should vary when $\alpha_i$ is replaced by $\alpha_i+r$.

\subsection{Definition of extended $r$-spin correlators}\label{subsection:definition of extended}

Suppose that $g=0$,
\begin{equation}
\label{eq:range}
\alpha_i \in \{-1,0,\ldots, r-1\}
\end{equation}
for each $i$, and there is at most one $i$ such that $\alpha_i=-1$.  In this case, one still has $R^0\pi_*\mathcal{S}=0$, so we can define $\mathcal{W}:= (R^1\pi_*\mathcal{S})^{\vee}$ and set
\[c_W^{\ext} := e((R^1\pi_*\mathcal{S})^{\vee}) = e(R^0\pi_*(\mathcal{S}^{\vee} \otimes \omega_{\pi})).\]
When there is no $i$ for which $\alpha_i=-1$, this simply recovers the definition of $c_W$ given above.  The same formula \eqref{eq:rankW} (with $g=0$) gives the complex codimension of $c_W^{\ext}$, and there are correlators
\[\left\langle \prod_{i=1}^l \tau_{d_i}^{\alpha_i} \right\rangle^{\frac{1}{r},\ext}_{0} := r\int_{\M_{0,\{\alpha_1, \ldots, \alpha_n\}}^{1/r}} e\left(\mathcal{W} \oplus \bigoplus_{i=1}^l \mathbb{L}_i^{\oplus d_i}\right)\]
that vanish unless the dimension condition \eqref{eq:when_int_is_number_closed_extended} is satisfied.  We refer to these as {\it extended $r$-spin correlators}.

\begin{rmk}
We caution the reader that the Ramond vanishing property does {\it not} hold for the extended $r$-spin correlators.  For example, a nonvanishing extended correlator with an insertion of twist $r-1$ is calculated in Lemma~\ref{lemma:small closed extended correlators}.
\end{rmk}

\subsection{Properties of the extended Witten class}

Our goal for this section is to prove that the genus-zero extended $r$-spin correlators satisfy certain equations analogous to the string equation and topological recursion relations in Gromov--Witten theory.  To do so, we must study how $c_W^{\ext}$ behaves under the inclusion of boundary divisors and forgetful morphisms.

Let us fix some notation.  In general, boundary strata in $\M_{g,\{\alpha_1, \ldots, \alpha_n\}}^{1/r}$ are indexed by certain decorated graphs, in which each vertex $v$ represents an irreducible component and is labeled with its genus $g(v)$, each half-edge $h$ represents a half-node and is labeled with its twist $\alpha(h)$, and there are $n$ numbered legs labeled with the twists $\alpha_1, \ldots, \alpha_n$.  We denote by $\vec{\alpha}(v)$ the tuple recording the twists at all half-edges incident to $v$, including the legs.  Note that the elements of~$\vec{\alpha}(v)$ lie in $\{-1,0,1,\ldots, r-1\}$ and that the twist is $-1$ at each half-node on which the isotropy group acts trivially on the fiber of $S$.

Given $\Gamma$ as above, let $\M_{\Gamma}^{1/r} \subset \M_{g,\{\alpha_1,\ldots,\alpha_n\}}^{1/r}$ be the boundary stratum consisting of $r$-spin curves with decorated dual graph $\Gamma$.  Let $\widetilde\Gamma$ be the disconnected graph obtained by cutting all of the edges of $\Gamma$, and let
\[\M^{1/r}_{\widetilde\Gamma} := \prod_{v \in V(\Gamma)} \M^{1/r}_{g(v), \vec{\alpha}(v)}\]
be the associated moduli space, where $V(\Gamma)$ is the vertex set of $\Gamma$.  Unlike the moduli space of curves, the $r$-spin moduli space does not in general have a gluing map $\M^{1/r}_{\widetilde\Gamma} \rightarrow \M_{\Gamma}^{1/r}$, because there is no canonical way to glue the fibers of $S$ at the nodes.  Nevertheless, letting $\M_{\Gamma}$ and~$\M_{\widetilde\Gamma}$ denote the moduli spaces of marked curves with dual graphs $\Gamma$ and $\widetilde\Gamma$, respectively, one has morphisms
\begin{gather}\label{eq:gluing sequence}
\M^{1/r}_{\widetilde\Gamma} \xleftarrow{p} \M_{\widetilde\Gamma} \times_{\M_{\Gamma}} \M_{\Gamma}^{1/r} \xrightarrow{\widetilde\mu} \M_{\Gamma}^{1/r}\xrightarrow{\widetilde{i}_{\Gamma}} \M_{g,\{\alpha_1,\ldots,\alpha_n\}}^{1/r}.
\end{gather}

Let $g=0$. Then for any decorated graph $\Gamma$ of genus $0$ we have $\oM_{\tGamma}=\oM_{\Gamma}$ and, therefore, $\oM_{\tGamma}\times_{\oM_{\Gamma}}\oM_{\Gamma}^{1/r}=\oM_{\Gamma}^{1/r}$. Then the map~$\widetilde{\mu}$ in~\eqref{eq:gluing sequence} is the identity. Suppose now that the twists~$\alpha_1, \ldots, \alpha_n$ at the legs of $\Gamma$ satisfy \eqref{eq:range} with at most one $i$ such that $\alpha_i=-1$, then we refer to $\Gamma$ as a {\it genus-zero extended $r$-spin dual graph}.  Note that even in this case, the vertex moduli spaces $\M^{1/r}_{0, \vec{\alpha}(v)}$ may not themselves fit into the extended framework, since there may be more than one half-edge of twist $-1$ incident to a given vertex.  Nevertheless, there is a consistent way to rectify the situation.  Indeed, there is a unique function
\[\alpha': H(\Gamma) \rightarrow \{-1,0,1,\ldots, r-1\}\]
on the half-edge set $H(\Gamma)$ such that
\begin{enumerate}
\item $\alpha'(h) \equiv \alpha(h) \mod r$, for all $h$,
\item $\alpha'(h) = \alpha(h)$, if $h$ is a leg,
\item for each edge $e = (h_1,h_2)$, such that $\alpha(h_1) = \alpha(h_2) = -1$, we have
\begin{enumerate}
\item $\alpha'(h_1)=\alpha'(h_2)=r-1$, if $\alpha_1,\ldots,\alpha_n\in\{0,\ldots,r-1\}$,
\item exactly one $i \in \{1,2\}$ with $\alpha'(h_i) = r-1$, if there exists $1\le j\le n$, such that $\alpha_j=-1$,
\end{enumerate}
\item for any edge $e$, each of the two connected components of the graph, obtained by cutting~$e$, has at most one leg $h$ for which $\alpha'(h) = -1$.
\end{enumerate}
Let
\begin{equation}
\label{eq:Mext}
\M_{\widetilde\Gamma}^{\ext} := \prod_{v \in V(\Gamma)} \M^{1/r}_{0, \vec{\alpha}'(v)},
\end{equation}
where $\vec{\alpha}'(v)$ is the tuple recording the values of $\alpha'(h)$ at all of the half-edges incident to $v$.  For each $v$, let $\mathcal{T}(v) \subset H(v)$ denote the subset of the incident half-edges for which $\alpha'(h) \neq \alpha(h)$.  Then there is a morphism $\tau_v: \M^{1/r}_{0, \vec{\alpha}(v)} \rightarrow \M^{1/r}_{0,\vec{\alpha}'(v)}$ defined by sending an $r$-spin structure $S$ with twists $\vec{\alpha}(v)$ to
\[S':= S \otimes \O\left(\sum_{h \in \mathcal{T}(v)} -[z_h]\right),\]
where $z_h$ is the marked point corresponding to $h$; it is straightforward to see that $S'$ is an $r$-spin structure with twists $\vec{\alpha}'(v)$.  The product of the morphisms $\tau_v$ defines
\[\tau: \M_{\widetilde\Gamma}^{1/r} \rightarrow \M^{\ext}_{\widetilde\Gamma},\]
so we now have
\[\M^{\ext}_{\widetilde\Gamma} \xleftarrow{q} \M_{\Gamma}^{1/r} \xrightarrow{\widetilde{i}_{\Gamma}} \M_{0,\{\alpha_1,\ldots,\alpha_n\}}^{1/r},\]
where $q:=\tau \circ p$.

We can now state the decomposition property of the extended Witten class along nodes.

\begin{lemma}
\label{lem:decomposition}
Let $\Gamma$ be a genus-zero extended $r$-spin dual graph with two vertices $v_1$ and $v_2$ connected by a single edge $e$.  Let $\M_1$ and $\M_2$ be the factors of $\M_{\widetilde\Gamma}^{\ext}$ corresponding to the two vertices, and let $c_{W_1}^{\ext}$ and $c_{W_2}^{\ext}$ be the Witten classes on these two moduli spaces.  Then
\begin{equation}
\label{eq:decomp}
q_*\widetilde{i}_{\Gamma}^*c_W^{\ext} = c_{W_1}^{\ext} \boxtimes c_{W_2}^{\ext}.
\end{equation}
\end{lemma}

Before proving the lemma, let us present an alternative way to construct the extended Witten class. Note that for each $1\le i\le n$ there is a canonical identification $\oM_{0,\{\alpha_1,\ldots,\alpha_n\}}^{1/r}=\oM_{0,\{\alpha_1,\ldots,\alpha_i+r,\ldots,\alpha_n\}}^{1/r}$ defined by sending an $r$-spin structure $S$ with twists $\alpha_1,\ldots,\alpha_n$ to
\begin{gather*}
S':= S \otimes \O\left(-[z_i]\right).
\end{gather*}
Moreover, this allows to identify the moduli spaces $\oM_{0,\{\alpha_1,\ldots,\alpha_n\}}^{1/r}$ and $\oM_{0,\{\beta_1,\ldots,\beta_n\}}^{1/r}$, if all the differences $\alpha_i-\beta_i$ are divisible by~$r$. Denote by $\mcS_{\alpha_1,\ldots,\alpha_n}\to\mcC$ the universal line bundle over the universal curve $\mcC\xrightarrow{\pi}\oM_{0,\{\alpha_1,\ldots,\alpha_n\}}^{1/r}$.

\begin{lemma}\label{lemma:alternative expression for extended}
For any $n\ge 2$ and an $n$-tuple $0\le\beta_1,\ldots,\beta_n\le r-1$, the extended Witten class~$c^\ext_W$ on $\oM^{1/r}_{0,\{-1,\beta_1,\ldots,\beta_n\}}$ is equal to
$$
c_W^\ext=c_{\rk-1}\left((R^1\pi_*\mcS_{r-1,\beta_1,\ldots,\beta_n})^\vee\right)\in H^*\left(\oM^{1/r}_{0,\{r-1,\beta_1,\ldots,\beta_n\}}\right)=H^*\left(\oM^{1/r}_{0,\{-1,\beta_1,\ldots,\beta_n\}}\right),
$$
where $\rk=\rk\left((R^1\pi_*\mcS_{r-1,\beta_1,\ldots,\beta_n})^\vee\right)=\frac{1+\sum\beta_i}{r}$.
\end{lemma}
\begin{proof}
Let us consider the exact sequence~\eqref{eq:vanishingles} over the moduli space~$\oM_{0,\{r-1,\beta_1,\ldots,\beta_n\}}^{1/r}$. Note that $\mcS=\mcS_{r-1,\beta_1,\ldots,\beta_n}$ and $\widetilde\mcS=\mcS_{-1,\beta_1,\ldots,\beta_n}$. Therefore, by multiplicativity of the total Chern class in~\eqref{eq:vanishingles},
$$
c_W^\ext=e\left((R^1\pi_*\mcS_{-1,\beta_1,\ldots,\beta_n})^\vee\right)=c_{\rk-1}\left((R^1\pi_*\mcS_{r-1,\beta_1,\ldots,\beta_n})^\vee\right),
$$
which proves the lemma.
\end{proof}

This lemma implies the following symmetry property of the extended Witten class.

\begin{cor}\label{corollary:symmetry of extended}
Suppose $n\ge 1$ and $0\le\beta_1,\ldots,\beta_n\le r-1$. Then, under the identification $\oM_{0,\{-1,r-1,\beta_1,\ldots,\beta_n\}}^{1/r}=\oM_{0,\{r-1,-1,\beta_1,\ldots,\beta_n\}}^{1/r}$, the extended Witten classes on these moduli spaces become equal. In other words,
\begin{gather}\label{eq:symmetry of extended}
e\left((R^1\pi_*\mcS_{-1,r-1,\beta_1,\ldots,\beta_n})^\vee\right)=e\left((R^1\pi_*\mcS_{r-1,-1,\beta_1,\ldots,\beta_n})^\vee\right).
\end{gather}
\end{cor}
\begin{proof}
By Lemma~\ref{lemma:alternative expression for extended}, both sides of~\eqref{eq:symmetry of extended} are equal to $c_{\rk-1}\left((R^1\pi_*\mcS_{r-1,r-1,\beta_1,\ldots,\beta_n})^\vee\right)$.
\end{proof}
\begin{proof}[Proof of Lemma~\ref{lem:decomposition}]
Over $\oM_{\Gamma}^{1/r}$, there are two universal curves: namely, we define $\mcC_{\Gamma}$ by the fiber diagram
\[
\xymatrix{
\mathcal{C}_{\Gamma} \ar[r]\ar[d]_{\pi} & \mathcal{C}\ar[d]^{\pi}\\
\M_{\Gamma}^{1/r} \ar[r]^-{\widetilde{i}_{\Gamma}} & \M_{0,\{\alpha_1,\ldots,\alpha_n\}}^{1/r},}
\]
and we define $\mathcal{C}_{\widetilde{\Gamma}}$ by the fiber diagram
\begin{gather*}
\xymatrix{
\widetilde{\mathcal{C}}\ar[d] & \mathcal{C}_{\widetilde{\Gamma}}\ar[l]\ar[d]^{\widetilde{\pi}}\\
\M_{\widetilde{\Gamma}}^{\ext} & \M_{\Gamma}^{1/r}\ar[l]_-{q},}
\end{gather*}
in which $\widetilde{\mathcal{C}}$ is the universal curve over $\M_{\widetilde{\Gamma}}^{\ext}$, induced by the product description \eqref{eq:Mext}. One can decompose $\mathcal{C}_{\widetilde{\Gamma}} = \mathcal{C}_1 \sqcup \mathcal{C}_2$, and we let $\pi_i = \widetilde{\pi}|_{\mathcal{C}_i}$, for $i=1,2$. Let
\[n: \mathcal{C}_{\widetilde{\Gamma}} \rightarrow \mathcal{C}_{\Gamma}\]
be the universal normalization morphism.

Denote (by a slight abuse of notation) the pullback to $\mathcal{C}_{\Gamma}$ of the universal line bundle on $\mathcal{C}$ by $\mathcal{S}$, and let $\mathcal{S}_i = n^*\mathcal{S}|_{\mathcal{C}_i}$ for $i=1,2$.  There is a normalization exact sequence
\[0 \rightarrow \mathcal{S} \rightarrow n_*n^*\mathcal{S} \rightarrow \mathcal{S}|_{\Delta_e} \rightarrow 0,\]
where $\Delta_e \subset \mathcal{C}_{\Gamma}$ picks out the node that corresponds to the edge $e$ of $\Gamma$.  The associated long exact sequence reads
\begin{equation}
\label{eq:normalization}
0 \rightarrow R^0\pi_{1*}\mathcal{S}_1 \oplus R^0\pi_{2*}\mathcal{S}_2 \rightarrow R^0\pi_*(\mcS|_{\Delta_e})\rightarrow R^1\pi_*\mathcal{S} \rightarrow  R^1\pi_{1*}\mathcal{S}_1 \oplus R^1\pi_{2*}\mathcal{S}_2 \rightarrow 0.
\end{equation}
From here, we break the argument into three cases.

First, suppose that one (hence both) of the half-edges of $e$ is Neveu--Schwarz.  Then $R^0\pi_*(\mcS|_{\Delta_e})=0$, since sections of an orbifold line bundle necessarily vanish at points where the isotropy group acts nontrivially on the fiber.  Thus, we have
\[
R^1\pi_*\mathcal{S} \cong R^1\pi_{1*}\mathcal{S}_1 \oplus R^1\pi_{2*}\mathcal{S}_2
\]
on $\M_{\Gamma}^{1/r}$.  This implies that $\widetilde{i}_{\Gamma}^* c_{W}^{\ext} = q^* (c_{W_1}^{\ext} \boxtimes c_{W_2}^{\ext})$, and the claim follows from the fact that $\deg1 q=1$ \cite[page~1348]{CZ}.

Next, suppose that one (hence both) of the half-edges of $e$ is Ramond.  Assume, without loss of generality, that the vertices~$v_1$ and~$v_2$ of~$\Gamma$ are incident to the legs marked by~$\{1,\ldots,n_1\}$ and $\{n_1+1,\ldots,n\}$, respectively. Consider the universal line bundles $\mcS_{\alpha_1,\ldots,\alpha_{n_1},-1}$ and $\mcS_{-1,\alpha_{n_1+1},\ldots,\alpha_n}$ over $\widetilde{\mcC}$, corresponding to the components $\oM_1$ and $\oM_2$, respectively. The pullbacks, via the map $\mcC_{\tGamma}\to\widetilde{\mcC}$, of these line bundles to $\mcC_1$ and $\mcC_2$, respectively, will be denoted by the same letters. Clearly,
$$
\mcS_1=\mcS_{\alpha_1,\ldots,\alpha_{n_1},-1},\qquad \mcS_2=\mcS_{-1,\alpha_{n_1+1},\ldots,\alpha_n}.
$$
Note that $R^0\pi_*(\mcS|_{\Delta_e}) = \sigma_e^*\mathcal{S}$, where $\sigma_e$ is the section of $\mathcal{C}_{\Gamma}$ with image $\Delta_e$.  Since
\[
(\sigma_e^*\mathcal{S})^{\otimes r} \cong \sigma_e^*\omega_{\pi,\log} \cong \O_{\M_{\Gamma}^{1/r}},
\]
we have $c_1(\sigma_e^*\mathcal{S}) = 0$.

Suppose, additionally, that $\alpha_1,\ldots,\alpha_n\in\{0,\ldots,r-1\}$. Then we have $R^0\pi_{1*}\mathcal{S}_1 = R^0\pi_{2*}\mathcal{S}_2 = 0$ and, by multiplicativity of the total Chern classes in \eqref{eq:normalization}, we find that
\begin{gather}\label{eq:relation between Chern classes}
c_i\left((R^1\pi_*\mcS)^\vee\right)=c_i\left((R^1\pi_{1*}\mcS_{\alpha_1,\ldots,\alpha_{n_1},-1})^\vee\oplus(R^1\pi_{2*}\mcS_{-1,\alpha_{n_1+1},\ldots,\alpha_n})^\vee\right),\quad i\ge 0.
\end{gather}
By the Ramond vanishing, the right-hand side of~\eqref{eq:decomp} is equal to zero. Note that $\rk(R^1\pi_*\mcS)=\rk(R^1\pi_{1*}\mcS_1\oplus R^1\pi_{2*}\mcS_2)+1$, therefore, by~\eqref{eq:relation between Chern classes}, $\widetilde{i}_{\Gamma}^*c_W^{\ext}=e((R^1\pi_*\mcS)^\vee)=0$.

Finally, suppose again that one (hence both) of the half-edges of $e$ is Ramond, but there is a marked point of twist $-1$. Without loss of generality, we can assume that $\alpha_1=-1$. Let us write equation~\eqref{eq:relation between Chern classes} for the $n$-tuple $r-1,\alpha_2,\ldots,\alpha_n$ and for $i=\rk\left(R^1\pi_*\mcS_{r-1,\alpha_2,\ldots,\alpha_n}\right)-1$. Then, by Lemma~\ref{lemma:alternative expression for extended}, the left-hand side is equal to $\ti_\Gamma^*c_W^\ext$. The right-hand side of~\eqref{eq:relation between Chern classes} is equal to
\begin{align*}
&e\left((R^1\pi_{1*}\mcS_{r-1,\alpha_2,\ldots,\alpha_{n_1},-1})^\vee\right)e\left((R^1\pi_{2*}\mcS_{-1,\alpha_{n_1+1},\ldots,\alpha_n})^\vee\right)\stackrel{\text{by Corollary~\ref{corollary:symmetry of extended}}}{=}\\
&\hspace{5cm}=e\left((R^1\pi_{1*}\mcS_{-1,\alpha_2,\ldots,\alpha_{n_1},r-1})^\vee\right)\cdot e\left((R^1\pi_{2*}\mcS_{-1,\alpha_{n_1+1},\ldots,\alpha_n})^\vee\right)=\\
&\hspace{5cm}=q^*(c_{W_1}^\ext\boxtimes c_{W_2}^\ext).
\end{align*}
Hence, $\ti_\Gamma^*c_W^\ext=q^*(c_{W_1}^\ext\boxtimes c_{W_2}^\ext)$ and using again that $\deg1 q=1$ \cite[page~1348]{CZ} we see that the lemma is proved.
\end{proof}

The other property of the extended Witten class that we require is its behavior under pullback via the forgetful map. There is only a forgetful map on the $r$-spin moduli space forgetting marked points whose twist equals zero.  Let
\[\text{For}_{n+1}\colon \M_{0, \{\alpha_1, \ldots, \alpha_n,0\}}^{1/r} \rightarrow \M_{0,\{\alpha_1, \ldots, \alpha_n\}}^{1/r}\]
be the map that forgets $z_{n+1}$ and its orbifold structure and stabilizes the curve $C$ as necessary.

\begin{lemma}
\label{lem:forgetful}
Let $c_W^{\mathrm{ext}}$ be the Witten class for $\M_{0, \{\alpha_1, \ldots, \alpha_n\}}^{1/r}$, and let $\widetilde{c}_W^{\mathrm{ext}}$ be the Witten class for $\M_{0, \{\alpha_1, \ldots, \alpha_n,0\}}^{1/r}$.  Then
\[\widetilde{c}_W^{\mathrm{ext}} = \mathrm{For}_{n+1}^*c_W^{\mathrm{ext}}.\]
\begin{proof}
This follows immediately from the fact that the universal line bundle pulls back under the induced morphism $\widetilde{\text{For}}_{n+1}: \mathcal{C}' \rightarrow \mathcal{C}$ on the universal curves.
\end{proof}
\end{lemma}

\subsection{TRRs and string equation in extended $r$-spin theory}

We are now prepared to state and prove the topological recursion relations satisfied by the extended $r$-spin correlators in genus zero.  These follow from Lemmas~\ref{lem:decomposition} and \ref{lem:forgetful}, by essentially the same argument as given by Jarvis--Kimura--Vaintrob in \cite{JKV}. We denote by $[n]$ the set $\{1,2,\ldots,n\}$.

\begin{lemma}\label{lem:trr}
For any $i$ with $d_i>0$ and any $j \neq k \in [n] \setminus \{i\}$, we have
\begin{equation}\label{eq:general_TRR}
\<\prod_{l\in [n]}\tau_{d_l}^{\alpha_l}\>^{\frac{1}{r},\mathrm{ext}}_0 =
\sum_{\substack{I\coprod J = [n]\setminus\{i\}\\j,k\in J}}\sum_{\alpha=-1}^{r-1}\left\langle\tau_0^{\alpha}\tau_{d_i-1}^{\alpha_i}\prod_{l\in I}\tau_{d_l}^{\alpha_l}\right\rangle^{\frac{1}{r},\mathrm{ext}}_0
\left\langle\tau_0^{r-2-\alpha}\prod_{l\in J}\tau_{d_l}^{\alpha_l}\right\rangle^{\frac{1}{r},\mathrm{ext}}_0 .
\end{equation}

In particular, we have the following equations, where $K \subset [n]$ is the set of marked points whose twist is $r-1$ and without loss of generality we can assume that $\alpha_1=-1$:

\begin{enumerate}
\item (Neveu--Schwarz TRR) For any $i\in [n]\backslash(K \cup \{1\})$ with $d_i > 0$ and any $j\in [n]\setminus \{i,1\}$ we have
\begin{align}\label{eq:NS_TRR}
\left\langle\prod_{l\in [n]}\tau_{d_l}^{\alpha_l}\right\rangle^{\frac{1}{r},\mathrm{ext}}_0=&\sum_{\substack{I\coprod J = [n]\setminus\{i\}\\\{1,j\}\cup K\subseteq J}}\sum_{\alpha=0}^{r-2}\left\langle\tau_0^{\alpha}\tau_{d_i-1}^{\alpha_i}\prod_{l\in I}\tau_{d_l}^{\alpha_l}\right\rangle^{\frac{1}{r}}_0\left\langle\tau_0^{r-2-\alpha}\prod_{l\in J}\tau_{d_l}^{\alpha_l}\right\rangle^{\frac{1}{r},\mathrm{ext}}_0+\\
&+\sum_{\substack{I\coprod J = [n]\setminus\{i\}\\1,j\in J}}\left\langle\tau_0^{-1}\tau_{d_i-1}^{\alpha_i}\prod_{l\in I}\tau_{d_l}^{\alpha_l}\right\rangle^{\frac{1}{r},\mathrm{ext}}_0\left\langle\tau_0^{r-1}\prod_{l\in J}\tau_{d_l}^{\alpha_l}\right\rangle^{\frac{1}{r},\mathrm{ext}}_0.\notag
\end{align}

\item (Ramond TRR) For any $i \in K$ with $d_i > 0$ and any $j \in [n]\setminus \{1,i\}$ we have
\begin{gather}\label{eq:Ramond_TRR}\left\langle\prod_{l\in [n]}\tau_{d_l}^{\alpha_l}\right\rangle^{\frac{1}{r},\mathrm{ext}}_0 =\sum_{\substack{I\coprod J = [n]\setminus\{i\}\\1,j\in J}}\left\langle\tau_0^{-1}\tau_{d_i-1}^{\alpha_i}\prod_{l\in I}\tau_{d_l}^{\alpha_l}\right\rangle^{\frac{1}{r},\mathrm{ext}}_0\left\langle\tau_0^{r-1}\prod_{l\in J}\tau_{d_l}^{\alpha_l}\right\rangle^{\frac{1}{r},\mathrm{ext}}_0.\end{gather}

\item ($-1$ TRR) Suppose that $d_1>0$. Then for any $j \neq k \in [n]\setminus \{1\}$ we have
\begin{align}\label{eq:-1_TRR}
\left\langle\prod_{l\in [n]}\tau_{d_l}^{\alpha_l}\right\rangle^{\frac{1}{r},\mathrm{ext}}_0
=&\sum_{\substack{I\coprod J = [n]\setminus\{1\}\\j,k\in J,\,K\subset I}}\sum_{\alpha=0}^{r-2}\left\langle\tau_0^{\alpha}\tau^{-1}_{d_1-1}\prod_{l\in I}\tau_{d_l}^{\alpha_l}\right\rangle^{\frac{1}{r},\mathrm{ext}}_0\<\tau_0^{r-2-\alpha}\prod_{l\in J}\tau_{d_l}^{\alpha_l}\>^{\frac{1}{r}}_0+\\
&+\sum_{\substack{I\coprod J = [n]\setminus\{1\}\\j,k\in J}}\left\langle\tau_0^{r-1}\tau^{-1}_{d_1-1}\prod_{l\in I}\tau_{d_l}^{\alpha_l}\right\rangle^{\frac{1}{r},\mathrm{ext}}_0\<\tau^{-1}_0\prod_{l\in J}\tau_{d_l}^{\alpha_l}\>^{\frac{1}{r},\mathrm{ext}}_0.\notag
\end{align}

\end{enumerate}
\begin{proof}
Observe that the right-hand side of \eqref{eq:general_TRR} is always defined.  Indeed, while it could be the case that one of the two correlators in a summand has two insertions of twist $-1$, it is always multiplied by a correlator with at least one insertion of twist $r-1$ and no insertions of twist $-1$; thus, by Ramond vanishing in the usual closed theory, that summand of \eqref{eq:general_TRR} vanishes.

The proof of \eqref{eq:general_TRR} follows \cite[Section 4.2]{JKV}.  Namely, on $\M_{0,n}$, we apply the relation
\[\psi_i = \sum_{\Gamma} i_{\Gamma*}(1),\]
where the sum is over all single-edged graphs on which tail $i$ is separated by the edge from tails~$j$ and~$k$ (see, for example, \cite{Zvonkine}).  From here, we use the commutative diagram
\begin{equation}
\label{eq:diagram}
\xymatrix{
& \M_{\Gamma}^{1/r} \ar[r]^{\tilde i_{\Gamma}}\ar[dd]_{\widetilde{\rho}}\ar[dl]_{q} & \M_{0,\{\alpha_1,\ldots,\alpha_n\}}^{1/r}\ar[dd]_{\rho} \\
\M_{\tilde\Gamma}^{\ext}\ar[dr]_{\rho'} & & \\
& \M_{\Gamma}\ar[r]^{i_{\Gamma}} & \M_{0,n}
}
\end{equation}
for each single-edged graph $\Gamma$. First of all, note that $i_{\Gamma}^*\rho_* = r \cdot \widetilde{\rho}_* \widetilde{i}_{\Gamma}^*$, which follows from the fact that $\deg1\rho=\frac{1}{r}$ and $\deg1 \widetilde{\rho}=\frac{1}{r^2}$. Combining this property with Lemma~\ref{lem:decomposition} shows that
\begin{align*}
r\cdot \rho_*(\psi_i \cap c_W^{\ext}) &= r\sum_{\Gamma} \rho_*\left( \rho^*i_{\Gamma*}(1) \cap c_W^{\ext}\right)=\\
&=r\sum_{\Gamma} i_{\Gamma*}(1) \cap \rho_*(c_W^{\ext})=\\
&=r\sum_{\Gamma} i_{\Gamma*}\left( i_{\Gamma}^* \rho_*(c_W^{\ext})\right)=\\
&=r^2\sum_{\Gamma} i_{\Gamma*}\widetilde{\rho}_*\widetilde{i}_{\Gamma}^*(c_W^{\ext})=\\
&=r^2\sum_{\Gamma} i_{\Gamma*}\rho'_*q_*\widetilde{i}_{\Gamma}^* (c_W^{\ext})=\\
&=r^2\sum_{\Gamma} i_{\Gamma*}\rho'_*\left(c_{W_1}^{\ext} \boxtimes c_{W_2}^{\ext}\right).
\end{align*}
Finally, multiplying this equation by the remaining $\psi$-classes and integrating proves the claim.  The other items are specializations of the general case, using that the extended theory agrees with the usual closed theory in the absence of an insertion of twist $-1$, and that in the usual closed theory Ramond vanishing holds.
\end{proof}
\end{lemma}

The string equation in the closed extended $r$-spin theory is exactly as usual:

\begin{lemma}
We have
\label{lem:str}
\[
\left\langle\tau_0^0\prod_{i\in [n]}\tau_{d_i}^{\alpha_i}\right\rangle^{\frac{1}{r},\mathrm{ext}}_0=
\begin{cases}
\sum_{\substack{i\in [n]\\d_i>0}}\left\langle\tau_{d_i-1}^{\alpha_i}\prod_{j\neq i}\tau_{d_j}^{\alpha_j}\right\rangle^{\frac{1}{r},\mathrm{ext}}_0,&\text{if $n\ge 3$},\\
\delta_{d_1,0}\delta_{d_2,0}\delta_{\alpha_1+\alpha_2,r-2},&\text{if $n=2$}.
\end{cases}
\]
\end{lemma}
\begin{proof}
The string equation can either be deduced algebraically from the topological recursions above, or geometrically, by mimicking the proof of the ordinary string equation on $\M_{0,n}$ and applying Lemma~\ref{lem:forgetful}.
\end{proof}

\subsection{Base cases}

In addition to the above relations, the proof that the extended $r$-spin correlators satisfy the equations for the wave function of the Gelfand--Dickey hierarchy requires two base cases.  We collect these simple correlators in the following lemma.

\begin{lemma}\label{lemma:small closed extended correlators}
We have
\[\<\tau^{-1}_0 \tau^{1}_0 \tau^{r-2}_0\>^{1/r,\mathrm{ext}}_0 = 1\]
and
\[\<\tau^{-1}_0 \tau^{1}_0 \tau^{r-1}_0 \tau^{r-1}_0\>^{1/r,\mathrm{ext}}_0 = -\frac{1}{r}.\]
\end{lemma}
\begin{proof}
In the first case, the moduli space is isomorphic to $B\Z_r$ and Witten's class has rank zero, so the claim is immediate.

In the second case, Witten's bundle has rank one, so
\[c_W^{\ext} = c_1((R^1\pi_*\mathcal{S})^{\vee}) = -c_1(R^1\pi_*\mathcal{S}) = -\text{ch}_1(R^1\pi_*\mathcal{S}) = \text{ch}_1(R\pi_*\mathcal{S}).\]
The latter can be calculated via Chiodo's Grothendieck--Riemann--Roch formula \cite{ChiodoTowards}.  In our situation, Chiodo's formula reads:
\begin{gather*}
c_W^{\ext} = \frac{B_2\left(\frac{1}{r}\right)}{2} \kappa_1 - \frac{B_2(0)}{2} \psi_1 - \frac{B_2\left(\frac{2}{r}\right)}{2} \psi_2 - \frac{B_2(1)}{2}\psi_3 - \frac{B_2(1)}{2} \psi_4 + 3 \frac{rB_2\left(\frac{1}{r}\right)}{2}\widetilde{i}_{\Gamma*}(1),
\end{gather*}
where
\[B_2(x) = x^2 -x + \frac{1}{6}\]
is the second Bernoulli polynomial and $\Gamma$ is any of the three one-edged graphs on $\M_{0,\{-1,1,r-1,r-1\}}^{1/r}$, all of which yield the same divisor $\widetilde{i}_{\Gamma*}(1)$. From here, the claim is immediate from the fact that
\[\int_{\M_{0,\{-1,1,r-1,r-1\}}^{1/r}} \kappa_1 = \int_{\M_{0,\{-1,1,r-1,r-1\}}^{1/r}} \psi_i = \frac{1}{r},\]
for each $i$, and
\[\int_{\M_{0,\{-1,1,r-1,r-1\}}^{1/r}} \widetilde{i}_{\Gamma*}(1) = \frac{1}{r^2},\]
in which the last integral is a consequence of the $\Z_r$ scaling and the additional $\Z_r$ ghost automorphisms on a nodal $r$-spin curve.
\end{proof}

%%%%%%%%%%%%%%%%%%%%%%%%%%%%%%%%%%%%%%%%%%%%%%%%%%%%%%%%%%%%%%%%%%%
%%%%%%%%%%%%%%%%%%%%%%%%%%%%%%%%%%%%%%%%%%%%%%%%%%%%%%%%%%%%%%%%%%%

\section{Closed extended theory: algebra}\label{sec:alg}

In this section we prove the main result of the paper, Theorem~\ref{theorem:closed extended and wave}, which describes the function $F^{\frac{1}{r},\ext}_0$ in terms of the $r$-th Gelfand-Dickey hierarchy.

Sections~\ref{subsection:KP hierarchy} and~\ref{subsection:Gelfand-Dickey reduction} contain the necessary background information on the KP hierarchy, its Gelfand--Dickey reduction, and their relation to the closed $r$-spin partition function. Then, in Section~\ref{subsection:special solution}, we consider a special solution of the system of differential equations for the wave function of the $r$-GD hierarchy and discuss its main properties. In Section~\ref{subsection:closed extended and special}, we prove Theorem~\ref{theorem:closed extended and wave}: the function $F^{\frac{1}{r},\ext}_0$ coincides with the genus-zero part of the special solution. As a consequence of this result, in Section~\ref{subsection:Lax operator and closed extended} we obtain a simple interpretation of the genus-zero part of the Lax operator of the $r$-GD hierarchy in terms of the potential~$F^{\frac{1}{r},\ext}_0$. Finally, in Section~\ref{subsection:main conjecture}, we propose a conjecture about the structure of the closed extended $r$-spin correlators in all genera.

\subsection{Brief review of the KP hierarchy}\label{subsection:KP hierarchy}

The material of this section is borrowed from the book~\cite{Dic03}.

Consider formal variables $T_i$ for $i\ge 1$. A {\it pseudo-differential operator}~$A$ is a Laurent series
$$
A=\sum_{n=-\infty}^m a_n(T_*,\eps)\d_x^n,
$$
where $m\in \mathbb{Z}$, $\d_x$ is considered as a formal variable, and $a_n(T_*,\eps)$ are formal power series in the variables $T_i$ with coefficients that are complex Laurent polynomials in $\eps$:
$$
a_n(T_*,\eps)\in\mbC[\eps,\eps^{-1}][[T_1,T_2,\ldots]].
$$
The non-negative and negative degree parts of the pseudo-differential operator $A$ are defined by
\begin{gather*}
A_+:=\sum_{n=0}^m a_n\d_x^n \quad\text{and}\quad A_-:=A-A_+,
\end{gather*}
and residue of the operator $A$ is defined by
$$
\res A:=a_{-1}.
$$
The Laurent series
$$
\widehat A(T_*,\eps,z):=\sum_{n=-\infty}^m a_n(T_*,\eps)z^n,
$$
in which $z$ is a formal variable, is called the {\it symbol} of the operator~$A$.

The space of pseudo-differential operators is endowed with a structure of a non-commutative associative algebra, in which the multiplication is defined by the formula
\begin{gather*}
\d_x^k\circ f:=\sum_{l=0}^\infty\frac{k(k-1)\ldots(k-l+1)}{l!}\frac{\d^lf}{\d x^l}\d_x^{k-l},
\end{gather*}
where $k\in\Z$, $f\in\mbC[\eps,\eps^{-1}][[T_1,T_2,\ldots]]$, and the variable $x$ is identified with $T_1$.  For any $r\ge 2$ and any pseudo-differential operator~$A$ of the form
$$
A=\d_x^r+\sum_{n=1}^\infty a_n\d_x^{r-n},
$$
there exists a unique pseudo-differential operator $A^{\frac{1}{r}}$ of the form
$$
A^{\frac{1}{r}}=\d_x+\sum_{n=0}^\infty \widetilde{a}_n\d_x^{-n},
$$
such that $\left(A^{\frac{1}{r}}\right)^r=A$.

Consider the pseudo-differential operator
$$
\mcL=\d_x+\sum_{i\ge 1}u_i\d_x^{-i},\quad u_i\in\mbC[\eps,\eps^{-1}][[T_1,T_2,\ldots]].
$$
The {\it KP hierarchy} is the following system of partial differential equations for the power series~$u_i$:
\begin{gather}\label{eq:KP hierarchy}
\frac{\d\mcL}{\d T_n}=\eps^{n-1}\left[\left(\mcL^n\right)_+,\mcL\right],\quad n\ge 1.
\end{gather}
For $n=1$, the equation is equivalent to
$$
\frac{\d u_i}{\d T_1}=\frac{\d u_i}{\d x},\quad i\ge 1,
$$
compatible with our identification of $x$ with $T_1$.

\begin{rmk}
The factor $\eps^{n-1}$ is not included usually in the definition of the KP hierarchy~\eqref{eq:KP hierarchy}. This rescaling is necessary, if we want to describe the function $\exp\left(F^{\frac{1}{r},c}(T_*,\eps)\right)$ as a tau-function of the KP hierarchy.
\end{rmk}

Suppose an operator $\mcL$ satisfies the system~\eqref{eq:KP hierarchy}. Then there exists a pseudo-differential operator~$P$ of the form
\begin{gather}\label{dressing operator}
P=1+\sum_{n\ge 1}p_n(T_*,\eps)\d_x^{-n},
\end{gather}
satisfying $\mcL=P\circ\d_x\circ P^{-1}$ and
\begin{gather}\label{eq:Sato-Wilson equations}
\frac{\d P}{\d T_n}=-\eps^{n-1}\left(\mcL^n\right)_-\circ P,\quad n\ge 1.
\end{gather}
The operator $P$ is called the {\it dressing operator}.

We can now introduce the notion of a tau-function. Denote by $G_z$ the shift operator, which acts on a power series $f\in\mbC[\eps,\eps^{-1}][[T_1,T_2,\ldots]]$ as follows:
\begin{gather*}
G_z(f)(T_1,T_2,T_3,\ldots):=f\left(T_1-\frac{1}{z},T_2-\frac{1}{2\eps z^2},T_3-\frac{1}{3\eps^2 z^3},\ldots\right).
\end{gather*}
Let $P=1+\sum_{n\ge 1}p_n(T_*,\eps)\d_x^{-n}$ be the dressing operator of some operator $\mcL$ satisfying the KP hierarchy~\eqref{eq:KP hierarchy}. Then there exists a series $\tau\in\mbC[\eps,\eps^{-1}][[T_1,T_2,T_3,\ldots]]$ with constant term $\left.\tau\right|_{T_i=0}=1$ for which
$$
\widehat P=\frac{G_z(\tau)}{\tau}.
$$
The series $\tau$ is called a {\it tau-function} of the KP hierarchy. The operator $\mcL$ can be reconstructed from the tau-function $\tau$ by the following formula:
$$
\res\mcL^n=\eps^{1-n}\frac{\d^2\log\tau}{\d T_1\d T_n},\quad n\ge 1.
$$

Another important object associated to a solution of the KP hierarchy is the {\it wave function} (also called the Baker--Akhiezer function). Let $P$ be the dressing operator of some operator $\mcL$ satisfying the KP hierarchy~\eqref{eq:KP hierarchy} and let $\tau$ be the tau-function. Let
$$
\xi(T_*,\eps,z):=\sum_{k\ge 1}T_k\eps^{k-1}z^k.
$$
The wave function is defined by
$$
w(T_*,\eps,z):=\hP\cdot e^\xi=\frac{G_z(\tau)}{\tau}e^\xi\in\mbC[\eps,\eps^{-1}][[T_1,T_2,\ldots]][[z,z^{-1}]].
$$
It satisfies the equations
\begin{gather}\label{eq:equations for wave}
\frac{\d w}{\d T_n}=\eps^{n-1}(\mcL^n)_+w,\quad n\ge 1.
\end{gather}

\subsection{Gelfand--Dickey reduction and the closed $r$-spin partition function}\label{subsection:Gelfand-Dickey reduction}

Let $r\ge 2$. It is easy to see that the equation
\begin{gather}\label{eq:Gelfand-Dickey reduction}
(\mcL^r)_-=0
\end{gather}
is invariant with respect to the flows of the KP hierarchy~\eqref{eq:KP hierarchy}. Therefore, it defines a reduction of the KP hierarchy that is called the {\it $r$-th Gelfand--Dickey hierarchy}. Let
$$
L:=\mcL^r=\d_x^r+\sum_{i=0}^{r-2}f_i\d_x^i.
$$
Then all the coefficients $u_i$ of the operator $L$ can be expressed in terms of the functions $f_0,f_1,\ldots,f_{r-2}$, and the Gelfand--Dickey hierarchy can be written as the following system of equations:
\begin{gather}\label{eq:Gelfand-Dickey hierarchy}
\frac{\d L}{\d T_n}=\eps^{n-1}[(L^{n/r})_+,L],\quad n\ge 1.
\end{gather}
Clearly, $\frac{\d L}{\d T_{dr}}=0$ for any $d\ge 1$.

Suppose $\mcL$ is a solution of the KP hierarchy satisfying the property~\eqref{eq:Gelfand-Dickey reduction}. Then the dressing operator $P$ and the tau-function $\tau$ can be chosen in such a way that $\frac{\d P}{\d T_{dr}}=0$ and $\frac{\d\tau}{\d T_{dr}}=0$ for any $d\ge 1$. In this case, the function $\tau$ is called a {\it tau-function} of the Gelfand--Dickey hierarchy.

Consider the generating series $F^{\frac{1}{r},c}(t^*_*,\eps)$ of the closed $r$-spin intersection numbers,
\begin{gather*}
F^{\frac{1}{r},c}(t^*_*,\eps)=\sum_{\substack{g\ge 0,\,n\ge 1\\2g-2+n>0}}\frac{\eps^{2g-2}}{n!}\sum_{\substack{0\le \alpha_1,\ldots,\alpha_n\le r-2\\d_1,\ldots,d_n\ge 0}}\<\tau_{d_1}^{\alpha_1}\cdots\tau_{d_n}^{\alpha_n}\>^{\frac{1}{r}}_g t^{\alpha_1}_{d_1}\cdots t^{\alpha_n}_{d_n}.
\end{gather*}
In~\cite{FSZ10}, it is proven that the exponent $\tau^{\frac{1}{r},c}:=e^{F^{\frac{1}{r},c}}$ becomes a tau-function of the Gelfand--Dickey hierarchy after the change of variables
\begin{gather}\label{eq:closed r-spin change of variables}
T_k=\frac{1}{(-r)^{\frac{3k}{2(r+1)}-\frac{1}{2}-d}k!_r}t^\alpha_d,\quad 0\le\alpha\le r-2,\quad d\ge 0,
\end{gather}
where $k=\alpha+1+rd$ and
$$
k!_r:=\prod_{i=0}^d(\alpha+1+ri).
$$
The corresponding solution $L$ of the Gelfand--Dickey system~\eqref{eq:Gelfand-Dickey hierarchy} satisfies the initial condition
\begin{gather}\label{eq:initial condition for L}
L|_{T_{\ge 2}=0}=\d_x^r+\eps^{-r}rx.
\end{gather}

Let us prove a simple lemma describing the lowest-degree term in $\eps$ of the operator $L$.
\begin{lemma}\label{lemma:properties of L}
Define a Poisson bracket $\{\cdot,\cdot\}$ in the ring $\mbC[[T_*]][z]$ by
\begin{align*}
&\{z,f\}:=\d_x f,        && f\in\mbC[[T_*]],\\
&\{f_1,f_2\}=\{z,z\}:=0, && f_1,f_2\in\mbC[[T_*]].
\end{align*}
Then:
\begin{enumerate}
\item The functions~$f_i$ for $0\le i\le r-2$ have the form
$$
f_i=\sum_{g\ge 0}f^{[g]}_i \eps^{i-r+g},\quad f^{[g]}_i\in\mbC[[T_*]].
$$
\item Denote
$$
L_0:=\d_x^r+\sum_{i=0}^{r-2}f_i^{[0]}\d_x^i.
$$
Then the operator $L_0$ is uniquely determined by the equations
$$
\res\hL_0^{n/r}=\frac{\d^2 F^{1/r,c}_0}{\d T_1\d T_n},\quad 1\le n\le r-1.
$$
\item We have
$$
\frac{\d\hL_0}{\d T_a}=\left\{\left(\hL_0^{\frac{a}{r}}\right)_+,\hL_0\right\}.
$$
\end{enumerate}
\end{lemma}
\begin{proof}
Introduce the notation
\begin{align*}
&w_i^{(j)}:=\eps^{1-i}\d_x^j\frac{\d^2 F^{\frac{1}{r},c}}{\d T_1\d T_i}, && 1\le i\le r-1, \quad j\ge 0,\\
&f_i^{(j)}:=\d_x^jf_i, && 0\le i\le r-2, \quad j\ge 0.
\end{align*}
Then it is easy to see that
\begin{gather}\label{eq:f to w change}
w_i=\frac{i}{r} f_{r-i-1}+P_i(f_{r-i}^{(*)},\ldots,f_{r-2}^{(*)}),
\end{gather}
where $P_i$ is a polynomial in $f_{r-i}^{(*)},\ldots,f_{r-2}^{(*)}$. Moreover, if we assign to $f_k^{(j)}$ degree $r-k+j$, then the polynomial $P_i$ is homogeneous of degree $i+1$. The transformation~\eqref{eq:f to w change} is clearly invertible, so we have
\begin{gather}\label{eq:w to f change}
f_i=\frac{r}{i} w_{r-i-1}+Q_i(w_1^{(*)},\ldots,w_{r-i-2}^{(*)}),
\end{gather}
where $Q_i$ is a homogeneous polynomial of degree $r-i$ if we assign to $w_j^{(k)}$ degree $j+1+k$. It remains to note that
$$
w_i=\eps^{-i-1}\frac{\d^2 F^{1/r,c}_0}{\d T_1\d T_i}+O(\eps^{-i})
$$
and parts (a) and (b) of the lemma become clear.

Let us prove part (c), again using the homogeneity argument. We have
$$
\left[\left(L^\frac{a}{r}\right)_+,L\right]=\sum_{i=0}^{r-2}R_i(f_*^{(*)})\d_x^i,
$$
where a polynomial $R_i$ has degree $a+r-i$. Let us assign to $f^{(j)}_i$ differential degree $j$ and express a polynomial $R_i$ as
$$
R_i(f_*^{(*)})=\sum_{j\ge 0}R_{i,j}(f_*^{(*)}),
$$
where a polynomial $R_{i,j}$ has differential degree~$j$. Clearly, $R_{i,0}=0$ and for $R_{i,1}$ we have the formula
$$
\sum_{i=0}^{r-2}R_{i,1}z^i=\left\{\left(\hL^{\frac{a}{r}}\right)_+,\hL\right\}.
$$
We have $R_{i,j}=O(\eps^{i-r-a+2})$, for $j\ge 2$, and
$$
R_{i,1}=\eps^{i-r-a+1}\widetilde{R}_i(T_*)+O(\eps^{i-r-a+2}),
$$
where
$$
\sum_{i=0}^{r-2}\widetilde{R}_i(T_*)z^i=\left\{\left(\hL_0^{\frac{a}{r}}\right)_+,\hL_0\right\}.
$$
Part (c) of the lemma is proved.
\end{proof}

\subsection{Special solution of the equations for the wave function}\label{subsection:special solution}

Let $L$ be the solution of the Gelfand--Dickey hierarchy~\eqref{eq:Gelfand-Dickey hierarchy} associated to the tau-function~$\tau^{\frac{1}{r},c}$. Let~$\Phi(T_*,\eps)$ be the unique solution of the system of equations
\begin{gather}\label{eq:equations for phi}
\frac{\d\Phi}{\d T_n}=\eps^{n-1}(L^{n/r})_+\Phi
\end{gather}
that satisfies the initial condition
$$
\left.\Phi\right|_{T_{\ge 2}=0}=1.
$$

\begin{rmk}
In the case $r=2$, the function $\Phi$ was first considered in~\cite{Bur16}, and the first author proved there that the logarithm of~$\Phi$ coincides with the generating series of the intersection numbers on the moduli space of Riemann surfaces with boundary. Properties of the function~$\Phi$ for general $r$ were first studied in~\cite{BY15}. Note that the system of equations~\eqref{eq:equations for phi} for the function~$\Phi$ coincides with the system of equations~\eqref{eq:equations for wave} for the wave function $w$ of the KP hierarchy. In~\cite{BY15}, the authors found an explicit formula for $\Phi$ in terms of the wave function~$w$.
\end{rmk}

Denote $\phi:=\log\Phi$ and consider the expansion
$$
\phi=\sum_{g\in\mbZ}\eps^{g-1}\phi_g,\quad \phi_g\in\mbC[[T_*]].
$$
Let
\begin{gather}\label{eq:r-1 change}
T_{mr}=\frac{1}{(-r)^{\frac{m(r-2)}{2(r+1)}}m!r^m}t^{r-1}_{m-1},\quad m\ge 1.
\end{gather}
Define correlators $\<\tau^{\alpha_1}_{d_1}\cdots\tau^{\alpha_n}_{d_n}\>^\phi_g$ by
$$
\<\tau^{\alpha_1}_{d_1}\cdots\tau^{\alpha_n}_{d_n}\>^\phi_g:=\left.\frac{\d^n\phi_g}{\d t^{\alpha_1}_{d_1}\cdots\d t^{\alpha_n}_{d_n}}\right|_{t^*_*=0},\quad 0\le\alpha_1,\ldots,\alpha_n\le r-1,\quad d_1,\ldots,d_n\ge 0.
$$
\begin{lemma}\label{lemma:properties of phi0}
The correlator $\<\tau_{d_1}^{\alpha_1}\cdots\tau_{d_n}^{\alpha_n}\>^\phi_g$ can be non-zero only if
\begin{align}
&g\ge 0,\label{eq:phicor,property 1}\\
&\sum_{i=1}^n\left(\frac{\alpha_i}{r}+d_i-1\right)=\frac{(r+1)(g-1)}{r}.\label{eq:phicor,property 2}
\end{align}
\end{lemma}
\begin{proof}
Let us prove property~\eqref{eq:phicor,property 1}. The function $\phi$ satisfies the following equations:
$$
\frac{\d\phi}{\d T_n}=\eps^{n-1}\frac{(L^\frac{n}{r})_+e^\phi}{e^\phi}.
$$
Therefore, it is sufficient to prove that for any function $\theta\in\mbC[\eps,\eps^{-1}][[T_*]]$ such that $\theta=O(\eps^{-1})$, we have $\eps^{n-1}\frac{(L^\frac{n}{r})_+e^{\theta}}{e^{\theta}}=O(\eps^{-1})$. From Lemma~\ref{lemma:properties of L} it follows that the operator $L^{\frac{n}{r}}$ has the form $L^{\frac{n}{r}}=\sum_{i\le n}R_i\d_x^i$, $R_i\in\mbC[\eps,\eps^{-1}][[T_*]]$, where $R_i=O(\eps^{i-n})$. By induction, it is easy to prove that
\begin{gather}\label{eq:derivative of exponent}
\frac{\d_x^i e^\theta}{e^\theta}=i!\sum_{\substack{m_1,m_2,\ldots\ge 0\\\sum jm_j=i}}\prod_{j\ge 1}\frac{(\d_x^j\theta)^{m_j}}{(j!)^{m_j}m_j!},\quad i\ge 0.
\end{gather}
Since $\theta=O(\eps^{-1})$, we have $\prod_{j\ge 1}(\d_x^j\theta)^{m_j}=O(\eps^{-i})$. Therefore, $R_i\frac{\d_x^ie^\theta}{e^\theta}=O(\eps^{-n})$ and
$$
\eps^{n-1}\sum_{i=0}^n R_i\frac{\d_x^ie^\theta}{e^\theta}=O(\eps^{-1}).
$$
Property~\eqref{eq:phicor,property 1} is proved.

Let us prove property~\eqref{eq:phicor,property 2}. Consider the linear differential operator
$$
O:=\sum_{\substack{0\le\alpha\le r-1\\d\ge 0}}\left(\frac{\alpha}{r}+d-1\right)t^\alpha_d\frac{\d}{\d t^\alpha_d}-\frac{r+1}{r}\eps\frac{\d}{\d \eps}.
$$
We have to prove that $O\phi=0$ or, equivalently, $O\Phi=0$. Using the Gelfand--Dickey equations~\eqref{eq:Gelfand-Dickey hierarchy}, it is easy to show that
$$
\left(z\frac{\d}{\d z}+O\right)\hL=r\hL.
$$
Then, similarly, using equations~\eqref{eq:equations for phi} one can show that $O\Phi=0$. The lemma is proved.
\end{proof}

\begin{lemma}\label{lemma:string equation for phi}
The function $\phi$ satisfies the string equation
\begin{gather}\label{eq:string equation for phi}
\left(\frac{\d}{\d T_1}-\sum_{i\ge 1}(i+r)T_{i+r}\frac{\d}{\d T_i}\right)\phi=r\eps^{-1}T_r.
\end{gather}
\end{lemma}
\begin{proof}
Equation~\eqref{eq:string equation for phi} was proved in~\cite{BY15}, but in order to make the paper more self-contained let us give a proof here. Denote by $O$ the operator in the brackets on the left-hand side of~\eqref{eq:string equation for phi}. Using the Gelfand--Dickey equations~\eqref{eq:Gelfand-Dickey hierarchy}, it is easy to show that $O L=\eps^{-r}r$. Then, similarly, using equations~\eqref{eq:equations for phi} one can show that $O\Phi=(r\eps^{-1}T_r)\Phi$.
\end{proof}

\subsection{Closed extended potential and the special solution}\label{subsection:closed extended and special}

Consider the generating series~$F^{\frac{1}{r},\ext}_0(t^*_*)$ of the closed extended $r$-spin intersection numbers in genus zero,
$$
F_0^{\frac{1}{r},\ext}(t^*_*)=\sum_{n\ge 2}\frac{1}{n!}\sum_{\substack{0\le\alpha_1,\ldots,\alpha_n\le r-1\\d_1,\ldots,d_n\ge 0}}\<\tau^{\alpha_1}_{d_1}\cdots\tau^{\alpha_n}_{d_n}\tau^{-1}_0\>^{\frac{1}{r},\ext}_0t^{\alpha_1}_{d_1}\cdots t^{\alpha_n}_{d_n}.
$$
The main result of our paper is the following theorem.
\begin{thm}\label{theorem:closed extended and wave}
We have
\begin{gather*}
F_0^{\frac{1}{r},\ext}(t^{\le r-2}_*,t^{r-1}_*)=\sqrt{-r}\phi_0\left(t^{\le r-2}_*,\frac{1}{\sqrt{-r}}t^{r-1}_*\right).
\end{gather*}
\end{thm}
Before proving the theorem, let us formulate several auxiliary statements.

\begin{lemma}\label{lemma:open equations in genus 0}
The function $\phi_0$ satisfies the equations
\begin{gather}\label{eq:open equations in genus 0}
\frac{\d\phi_0}{\d T_n}=\left.(\hL_0^{\frac{n}{r}})_+\right|_{z=(\phi_0)_x},\quad n\ge 1.
\end{gather}
\end{lemma}
\begin{proof}
From Lemma~\ref{lemma:properties of L} it follows that the operator $L^{\frac{n}{r}}$ has the form $L^{\frac{n}{r}}=\sum_{i\le n}R_i\d_x^i$, where $R_i=\sum_{j\ge 0}R_{i,j}\eps^{i-n+j}$, $R_{i,j}\in\mbC[[T_*]]$. It is also clear that $\sum_{i\le n}R_{i,0}z^i=\hL_0^{\frac{n}{r}}$. We have
$$
\frac{\d\phi}{\d T_n}=\eps^{n-1}\sum_{i=0}^n R_i\frac{\d_x^ie^\phi}{e^\phi}.
$$
Since $\phi=\eps^{-1}\phi_0+O(\eps^0)$, for any numbers $m_1,m_2,\ldots\ge 0$, satisfying $\sum jm_j=i$, we have
$$
\prod_{j\ge 1}(\d_x^j\phi)^{m_j}=
\begin{cases}
\eps^{-i}(\phi_0)_x^i+O(\eps^{-i+1}),&\text{if $m_1=i$ and $m_2=m_3=\ldots=0$},\\
O(\eps^{-i+1}),&\text{otherwise}.
\end{cases}
$$
Using formula~\eqref{eq:derivative of exponent} we then get $R_i\frac{\d_x^ie^\phi}{e^\phi}=R_{i,0}\eps^{-n}(\phi_0)_x^i+O(\eps^{-n+1})$ and
$$
\eps^{n-1}\sum_{i=0}^n R_i\frac{\d_x^ie^\phi}{e^\phi}=\eps^{-1}\sum_{i=0}^n R_{i,0}(\phi_0)_x^i+O(\eps^0)=\eps^{-1}\left.(\hL_0^{\frac{n}{r}})_+\right|_{z=(\phi_0)_x}+O(\eps^0).
$$
The lemma is proved.
\end{proof}

\begin{prop}\label{prop:open TRR for phi0}
The function $\phi_0$ satisfies the equations
\begin{gather}\label{eq:open TRR for phi0}
\frac{\d^2\phi_0}{\d t^\alpha_{p+1}\d t^\beta_q}=\sum_{\mu+\nu=r-2}\frac{\d^2 F^{\frac{1}{r},c}_0}{\d t^\alpha_p\d t^\mu_0}\frac{\d^2\phi_0}{\d t^\nu_0\d t^\beta_q}+\frac{\d\phi_0}{\d t^\alpha_p}\frac{\d^2\phi_0}{\d t^{r-1}_0\d t^\beta_q},\quad 0\le\alpha,\beta\le r-1,\quad p,q\ge 0.
\end{gather}
\end{prop}
\begin{proof}
In the variables $T_i$ equations~\eqref{eq:open TRR for phi0} look as follows:
\begin{gather*}
d\left(\frac{\d\phi_0}{\d T_{a+r}}\right)=\sum_{b=1}^{r-1}\frac{a+r}{b(r-b)}\frac{\d^2 F^{\frac{1}{r},c}_0}{\d T_a\d T_b}d\left(\frac{\d\phi_0}{\d T_{r-b}}\right)+\frac{a+r}{r}\frac{\d\phi_0}{\d T_a}d\left(\frac{\d\phi_0}{\d T_r}\right),\quad a\ge 1.
\end{gather*}
By Lemma~\ref{lemma:open equations in genus 0}, this equation follows from
\begin{gather}\label{eq:first TRR}
d\left(\hL_0^{\frac{a+r}{r}}\right)_+=\sum_{b=1}^{r-1}\frac{a+r}{b(r-b)}\frac{\d^2 F^{\frac{1}{r},c}_0}{\d T_a\d T_b}d\left(\hL_0^{\frac{r-b}{r}}\right)_++\frac{a+r}{r}\left(\hL_0^{\frac{a}{r}}\right)_+d\hL_0,\quad a\ge 1.
\end{gather}

Let
$$
v_i:=\res\left(\hL_0^{\frac{i}{r}}\right),\quad 1\le i\le r-1.
$$
Clearly, the functions $f^{[0]}_0,\ldots,f^{[0]}_{r-2}$ and $v_1,\ldots,v_{r-1}$ are related by an invertible polynomial transformation. Therefore, for any $i\in\mbZ$, the coefficients of the Laurent series $\hL_0^i$ can be considered as polynomials in $v_1,\ldots,v_{r-1}$. For $a\ge 1$ and $1\le b\le r-1$ we have the following identity (see, for example, \cite[Sections 3.5 and 3.6]{Dic03}):
$$
\frac{\d^2 F^{\frac{1}{r},c}_0}{\d T_a\d T_b}=\frac{b(r-b)}{a+r}\frac{\d}{\d v_{r-b}}\res\hL_0^{\frac{a+r}{r}}.
$$
We see that equation~\eqref{eq:first TRR} is equivalent to
\begin{gather}\label{eq:proof of first TRR1}
d\left(\hL^\frac{a+r}{r}_0\right)_+=\sum_{b=1}^{r-1}\frac{\d}{\d v_b}\res\left(\hL_0^\frac{a+r}{r}\right)d\left(\hL^\frac{b}{r}_0\right)_+ +\underline{\frac{a+r}{r}\left(\hL_0^\frac{a}{r}\right)_+ d \hL_0}.
\end{gather}
Let us express the left-hand side in the following way:
$$
d\left(\hL_0^\frac{a+r}{r}\right)_+=\underline{\frac{a+r}{r}\left(\hL_0^\frac{a}{r}\right)_+d\hL_0}+\frac{a+r}{r}\left(\left(\hL_0^\frac{a}{r}\right)_-d\hL_0\right)_+.
$$
Notice that the underlined term here cancels the underlined term on the right-hand side of~\eqref{eq:proof of first TRR1}. Therefore, equation~\eqref{eq:proof of first TRR1} is equivalent to
\begin{gather}\label{eq:proof of first TRR2}
\frac{a+r}{r}\left(\left(\hL_0^\frac{a}{r}\right)_-d\hL_0\right)_+=\sum_{b=1}^{r-1}\frac{\d}{\d v_b}\res\left(\hL_0^\frac{a+r}{r}\right)d\left(\hL^\frac{b}{r}_0\right)_+.
\end{gather}

We compute
$$
\sum_{b=1}^{r-1}\frac{\d}{\d v_b}\res\left(\hL_0^\frac{a+r}{r}\right)d\left(\hL_0^\frac{b}{r}\right)_+=\sum_{b=1}^{r-1}\frac{b}{r}\frac{\d}{\d v_b}\res\left(\hL_0^\frac{a+r}{r}\right)\left(\hL_0^\frac{b-r}{r}d\hL_0\right)_+.
$$
We see that equation~\eqref{eq:proof of first TRR2} follows from the property
\begin{gather}\label{eq:proof of first TRR3}
\frac{a+r}{r}\left(\hL_0^\frac{a}{r}\right)_--\sum_{b=1}^{r-1}\frac{b}{r}\frac{\d}{\d v_b}\res\left(\hL_0^\frac{a+r}{r}\right)\hL_0^\frac{b-r}{r}\in z^{-r-1}\mbC[f^{[0]}_0,\ldots,f^{[0]}_{r-2}][[z^{-1}]].
\end{gather}
Recall that (see, for example, \cite[Section 3.5]{Dic03})
\begin{gather}\label{eq:genus 0 identity from Dickey}
\frac{a+r}{r}\left(\hL_0^\frac{a}{r}\right)_--\sum_{i=0}^{r-2}\frac{\d}{\d f^{[0]}_i}\res\left(\hL_0^\frac{a+r}{r}\right)z^{-i-1}\in z^{-r-1}\mbC[f^{[0]}_0,\ldots,f^{[0]}_{r-2}][[z^{-1}]].
\end{gather}
For any two elements $f,g\in\mbC[f^{[0]}_0,\ldots,f^{[0]}_{r-2}][[z^{-1}]]$ let us write $f\equiv g$ if the difference $f-g$ lies in $z^{-r-1}\mbC[f^{[0]}_0,\ldots,f^{[0]}_{r-2}][[z^{-1}]]$. Then, using identity~\eqref{eq:genus 0 identity from Dickey}, we can compute
\begin{align*}
\frac{a+r}{r}\left(\hL_0^\frac{a}{r}\right)_-\equiv&\sum_{i=0}^{r-2}\frac{\d}{\d f^{[0]}_i}\res\left(\hL_0^\frac{a+r}{r}\right)z^{-i-1}\equiv\sum_{b=1}^{r-1}\sum_{i=0}^{r-2}\frac{\d}{\d v_b}\res\left(\hL_0^\frac{a+r}{r}\right)\frac{\d v_b}{\d f^{[0]}_i}z^{-i-1}\equiv\\
\equiv&\sum_{b=1}^{r-1}\frac{b}{r}\frac{\d}{\d v_b}\res\left(\hL_0^\frac{a+r}{r}\right)\left(\hL_0^{\frac{b-r}{r}}\right)_-\equiv\sum_{b=1}^{r-1}\frac{b}{r}\frac{\d}{\d v_b}\res\left(\hL_0^\frac{a+r}{r}\right)\hL_0^{\frac{b-r}{r}}.
\end{align*}
Formula~\eqref{eq:proof of first TRR3} is proved. This completes the proof of the proposition.
\end{proof}

\begin{lemma}\label{lemma:initial condition for special}
We have
\begin{gather}\label{eq:initial conditions for special}
\<\tau^1_0\tau^{r-2}_0\>^\phi_0=
\begin{cases}
1,&\text{if $r=2$},\\
\frac{1}{\sqrt{-r}},&\text{if $r\ge 3$},
\end{cases}
\qquad
\<\tau^1_0(\tau^{r-1}_0)^2\>^\phi_0=
\begin{cases}
1,&\text{if $r=2$},\\
\frac{1}{\sqrt{-r}},&\text{if $r\ge 3$}.
\end{cases}
\end{gather}
\end{lemma}
\begin{proof}
Let us prove the first equation in~\eqref{eq:initial conditions for special}. We compute
\begin{gather*}
\left.\frac{\d^2\phi_0}{\d T_2\d T_{r-1}}\right|_{T_*=0}=\left.\frac{\d}{\d T_{r-1}}\left((\phi_0)^2_x+\frac{2}{r}f^{[0]}_{r-2}\right)\right|_{T_*=0}=\frac{2}{r}\left.\frac{\d f^{[0]}_{r-2}}{\d T_{r-1}}\right|_{T_*=0}.
\end{gather*}
We proceed as follows:
\begin{align*}
\left.\frac{\d\hL_0}{\d T_{r-1}}\right|_{T_*=0}=&\left.\left\{\left(\hL_0^{\frac{r-1}{r}}\right)_+,\hL_0\right\}\right|_{T_*=0}=\left.\left\{\left((z^r+rx)^{\frac{r-1}{r}}\right)_+,z^r+rx\right\}\right|_{x=0}\\
=&\left.\left\{z^{r-1},z^r+rx\right\}\right|_{x=0}=r(r-1)z^{r-2}.
\end{align*}
Therefore, $\left.\frac{\d^2\phi_0}{\d T_2\d T_{r-1}}\right|_{T_*=0}=2(r-1)$. Applying changes of variables~\eqref{eq:closed r-spin change of variables} and~\eqref{eq:r-1 change}, we see that the first equation in~\eqref{eq:initial conditions for special} is true.

Now, let us prove the second equation in~\eqref{eq:initial conditions for special}. First of all, we compute
$$
\left.\frac{\d^2\phi_0}{\d T_1\d T_r}\right|_{T_*=0}=\left.\d_x((\phi_0)_x^r+rx)\right|_{T_*=0}=r.
$$
Therefore, we have
$$
\left.\frac{\d^3\phi_0}{\d T_2\d T_r^2}\right|_{T_*=0}=\left.\frac{\d^2}{\d T_r^2}\left((\phi_0)_x^2+\frac{2}{r}f^{[0]}_{r-2}\right)\right|_{T_*=0}=\left.2\left(\frac{\d^2\phi_0}{\d T_1\d T_r}\right)^2\right|_{T_*=0}=2r^2.
$$
Using equations~\eqref{eq:closed r-spin change of variables} and~\eqref{eq:r-1 change}, we can see that the second equation in~\eqref{eq:initial conditions for special} is also true.
\end{proof}

\begin{proof}[Proof of Theorem~\ref{theorem:closed extended and wave}]
We have seen (see Section~\ref{subsection:definition of extended} and Lemmas~\ref{lem:trr},~\ref{lemma:small closed extended correlators}) that the function~$F^{\frac{1}{r},\ext}_0$ and the correlators $\<\tau_0^{-1}\tau^{\alpha_1}_{d_1}\cdots\tau^{\alpha_n}_{d_n}\>^{\frac{1}{r},\ext}_0$ satisfy the following properties:
\begin{align}
&\<\tau_0^{-1}\tau^{\alpha_1}_{d_1}\cdots\tau^{\alpha_n}_{d_n}\>^{\frac{1}{r},\ext}_0=0\quad\text{unless}\quad\frac{\sum\alpha_i-(r-1)}{r}+\sum d_i=n-2,\label{eq:main property 1}\\
&\frac{\d^2F^{\frac{1}{r},\ext}_0}{\d t^\alpha_{p+1}\d t^\beta_q}=\sum_{\mu+\nu=r-2}\frac{\d^2 F^{\frac{1}{r},c}_0}{\d t^\alpha_p\d t^\mu_0}\frac{\d^2F^{\frac{1}{r},\ext}_0}{\d t^\nu_0\d t^\beta_q}+\frac{\d F^{\frac{1}{r},\ext}_0}{\d t^\alpha_p}\frac{\d^2F^{\frac{1}{r},\ext}_0}{\d t^{r-1}_0\d t^\beta_q},\quad0\le\alpha,\beta\le r-1,\quad p,q\ge 0,\label{eq:main property 2}\\
&\<\tau_0^{-1}\tau_0^1\tau_0^{r-2}\>^{\frac{1}{r},\ext}_0=1,\qquad \<\tau_0^{-1}\tau_0^1(\tau_0^{r-1})^2\>^{\frac{1}{r},\ext}_0=-\frac{1}{r}.\label{eq:main property 3}
\end{align}
It is easy to see that for any non-zero complex constant $C$, the function $\tphi_0$ defined by
$$
\tphi_0(t^{\le r-2}_*,t^{r-1}_*):=C\phi_0(t^{\le r-2}_*,C^{-1}t^{r-1}_*)
$$
also satisfies equations~\eqref{eq:open TRR for phi0}. Let $C=\sqrt{-r}$ and
$$
\<\tau^{\alpha_1}_{d_1}\cdots\tau^{\alpha_n}_{d_n}\>^{\tphi}_0:=\left.\frac{\d^n\tphi_0}{\d t^{\alpha_1}_{d_1}\cdots\d t^{\alpha_n}_{d_n}}\right|_{t^*_*=0}.
$$
Lemmas~\ref{lemma:properties of phi0},~\ref{lemma:initial condition for special} and Proposition~\ref{prop:open TRR for phi0} imply that the function $\tphi_0$ and the correlators $\<\tau^{\alpha_1}_{d_1}\cdots\tau^{\alpha_n}_{d_n}\>^{\tphi}_0$ also satisfy properties~\eqref{eq:main property 1} -- \eqref{eq:main property 3}. Therefore, it is sufficient to check that these properties are enough to reconstruct the function~$F^{\frac{1}{r},\ext}_0$.

Note that the dimension constraint in property~\eqref{eq:main property 1} implies that $n\ge 2$. Therefore, by~\eqref{eq:main property 2}, it suffices to determine the primary correlators
\begin{gather}\label{primary closed extended}
\<\tau^{-1}_0\tau^{\alpha_1}_0\cdots\tau^{\alpha_l}_0(\tau^{r-1}_0)^k\>^{\frac{1}{r},\ext}_0,\quad 0\le\alpha_1,\ldots,\alpha_l\le r-2.
\end{gather}
By~\eqref{eq:main property 1}, such a correlator can be non-zero only if
\begin{gather}\label{eq:dimension condition for primary}
\sum_{i=1}^l(r-\alpha_i)+k=r+1.
\end{gather}

First of all, suppose that $l\le 1$, so we consider the correlators
$$
X_\alpha:=\<\tau^{-1}_0\tau^\alpha_0(\tau^{r-1}_0)^{\alpha+1}\>^{\frac{1}{r},\ext}_0,\quad 0\le\alpha\le r-1.
$$
Suppose $0\le\gamma\le r-2$ and consider the correlator
\begin{gather*}
Y_\gamma:=\<\tau^{-1}_0\tau_1^1\tau^\gamma_0(\tau^{r-1}_0)^{\gamma+2}\>^{\frac{1}{r},\ext}_0.
\end{gather*}
Let us compute it using the topological recursion relation~\eqref{eq:main property 2} in two different ways. Applying~\eqref{eq:main property 2} with $\alpha=1$, $p=q=0$ and $\beta=\gamma$, we obtain
\begin{gather}\label{eq:one-point, left-hand side}
Y_\gamma={\gamma+2\choose 2}X_1 X_\gamma.
\end{gather}
On the other hand, applying~\eqref{eq:main property 2} with $\alpha=1$, $p=q=0$ and $\beta=r-1$, we get
\begin{gather}\label{eq:one-point, right-hand side}
Y_\gamma={\gamma+1\choose 2}X_1X_\gamma+\<\tau^1_0\tau^\gamma_0\tau^{r-3-\gamma}_0\>^{\frac{1}{r},\ext}_0 X_{\gamma+1}.
\end{gather}
For $\gamma\le r-3$, we have $\<\tau^1_0\tau^\gamma_0\tau^{r-3-\gamma}_0\>^{\frac{1}{r}}_0=1$ (see e.g.~\cite[Section~0.6]{PPZ}), and, by~\eqref{eq:main property 3}, $\<\tau^1_0\tau^{r-2}_0\tau^{-1}_0\>^{\frac{1}{r},\ext}_0=1$. Therefore, equating the right-hand sides of~\eqref{eq:one-point, left-hand side} and~\eqref{eq:one-point, right-hand side}, we obtain
$$
X_{\gamma+1}=(\gamma+1)X_1X_\gamma,\quad 0\le\gamma\le r-2.
$$
Since $X_1=-\frac{1}{r}$, this immediately implies that $X_\alpha=(-1)^\alpha\frac{\alpha!}{r^\alpha}$, for any $0\le\alpha\le r-1$.

Consider now the correlator~\eqref{primary closed extended} with $l\ge 2$. Suppose that condition~\eqref{eq:dimension condition for primary} is satisfied. Recall that by~$[l]$ we denote the set $\{1,2,\ldots,l\}$. For a subset $I\subset[l]$, let
\begin{gather*}
m_I:=r+1-\sum_{i\in I}(r-\alpha_i),\qquad A_I:=\<\tau^{-1}_0\left(\prod_{i\in I}\tau^{\alpha_i}_0\right)(\tau^{r-1}_0)^{m_I}\>^{\frac{1}{r},\ext}_0.
\end{gather*}
Consider the correlator
\begin{gather}\label{eq:auxiliary correlator for l>=2}
Z:=\<\tau^{-1}_0\tau_1^{\alpha_1}\left(\prod_{i=2}^l\tau^{\alpha_i}_0\right)(\tau^{r-1}_0)^{k+r}\>^{\frac{1}{r},\ext}_0.
\end{gather}
Applying~\eqref{eq:main property 2} with $\alpha=\alpha_1$, $p=q=0$ and $\beta=\alpha_l$, we get
\begin{gather}\label{proof,primary,extended closed:left-hand side}
Z=\sum_{\substack{I\sqcup J=[l]\\1\in I,\,l\in J}}\sum_{\mu+\nu=r-2}\<\left(\prod_{i\in I}\tau^{\alpha_i}_0\right)\tau^\mu_0\>^{\frac{1}{r}}_0\underline{\<\tau^{-1}_0\tau^\nu_0\left(\prod_{j\in J}\tau^{\alpha_j}_0\right)(\tau^{r-1}_0)^{k+r}\>^{\frac{1}{r},\ext}_0}+\sum_{\substack{I\sqcup J=[l]\\1\in I,\,l\in J}}{r+k\choose m_I}A_I A_J.
\end{gather}
Note that the underlined term on the right-hand side of this equation vanishes, because otherwise we should have
\begin{multline*}
r-\nu+\sum_{j\in J}(r-\alpha_j)+k+r=r+1\Rightarrow\sum_{i\in I}(r-\alpha_i)=r+(r-\nu)\Rightarrow\\\Rightarrow\sum_{i\in I}(r-\alpha_i)\ge r+2\Rightarrow\sum_{i=1}^l(r-\alpha_i)\ge r+2,
\end{multline*}
which contradicts~\eqref{eq:dimension condition for primary}. On the other hand, applying to the correlator~\eqref{eq:auxiliary correlator for l>=2} relation~\eqref{eq:main property 2} with $\alpha=\alpha_1$, $p=q=0$, and $\beta=r-1$, we obtain
\begin{align}
Z=&\sum_{\substack{I\sqcup J=[l]\\1\in I}}{r+k-1\choose m_I}A_IA_J=\notag\\
=&\sum_{\substack{I\sqcup J=[l]\\1\in I,\,l\in J}}{r+k-1\choose m_I}A_IA_J+\sum_{\substack{I\sqcup J=[l]\\1,l\in I,\,J\ne\emptyset}}{r+k-1\choose m_I}A_IA_J+{r+k-1\choose k}A_{[l]}\<\tau^{-1}_0(\tau^{r-1}_0)^{r+1}\>^{\frac{1}{r},\ext}_0.\label{proof,primary,extended closed:right-hand side}
\end{align}
Note that here, by the same argument as above, we also do not have terms with closed $r$-spin correlators. Equating the right-hand side of~\eqref{proof,primary,extended closed:left-hand side} and expression~\eqref{proof,primary,extended closed:right-hand side} and using that $\<\tau^{-1}_0(\tau^{r-1}_0)^{r+1}\>^{\frac{1}{r},\ext}_0=(-1)^{r-1}\frac{(r-1)!}{r^{r-1}}$, we get
\begin{gather}\label{recursion for primary extended closed}
(-1)^{r-1}\frac{(r+k-1)!}{k!r^{r-1}}A_{[l]}=\sum_{\substack{I\sqcup J=[l]\\1\in I,\,l\in J}}{r+k-1\choose m_I-1}A_IA_J-\sum_{\substack{I\sqcup J=[l]\\1,l\in I,\,J\ne\emptyset}}{r+k-1\choose m_I}A_IA_J.
\end{gather}
We see that for $l\ge 2$, this equation allows one to compute the primary correlator~\eqref{primary closed extended} in terms of primary correlators with smaller~$l$. The theorem is proved.
\end{proof}

\subsection{Lax operator and the closed extended potential}\label{subsection:Lax operator and closed extended}

Let us show that the operator~$L_0|_{t^*_{\ge 1}=0}$ has a simple interpretation in terms of the function $F^{\frac{1}{r},\ext}_0$. Define the primary closed extended potential in genus $0$ by
$$
\mathcal{F}^{\frac{1}{r},\ext}_0(t^0_0,\ldots,t^{r-1}_0):=\left.F_0^{\frac{1}{r},\ext}\right|_{t^*_{\ge 1}=0}.
$$
\begin{prop}
We have
\begin{gather}\label{eq:extended and Lax}
\frac{\d\mathcal{F}^{\frac{1}{r},\ext}_0}{\d t^{r-1}_0}=\frac{1}{(-r)^{\frac{r-2}{2(r+1)}}r}\left.\hL_0\right|_{\substack{z=(-r)^{\frac{1-2r}{2(r+1)}}t^{r-1}_0\\t^*_{\ge 1}=0}}.
\end{gather}
\end{prop}
\begin{proof}
Using~Theorem~\ref{theorem:closed extended and wave} we see that equation~\eqref{eq:extended and Lax} is equivalent to the equation
\begin{gather*}
\left.\frac{\d\phi_0}{\d T_r}\right|_{T_{\ge r+1}=0}=\left.\hL_0\right|_{\substack{z=rT_r\\T_{\ge r+1}=0}},
\end{gather*}
that follows from Lemma~\ref{lemma:open equations in genus 0} and the fact that, by the string equation~\eqref{eq:string equation for phi}, $\left.(\phi_0)_x\right|_{T_{\ge r+1}=0}=rT_r$.
\end{proof}

\subsection{Main conjecture}\label{subsection:main conjecture}

We conjecture that for any $b\ge 1$, $g,n\ge 0$, $0\le\alpha_1,\ldots,\alpha_n\le r-1$, and $d_1,\ldots,d_n\ge 0$, there is a geometric construction of correlators
$$
\<\tau^{\alpha_1}_{d_1}\cdots\tau^{\alpha_n}_{d_n}(\tau^{-1}_0)^b\>^{\frac{1}{r},\ext}_g,
$$
generalizing the construction for $b=1$ and $g=0$.  Given such correlators, define a generating series~$F_g^{\frac{1}{r},\ext}(t^*_*)$ for any $g \geq 0$ by
$$
F_g^{\frac{1}{r},\ext}(t^*_*):=\sum_{\substack{h\ge 0,\,b\ge 1\\2h+b-1=g}}\sum_{n\ge 0}\frac{1}{b!n!}\sum_{\substack{0\le\alpha_1,\ldots,\alpha_n\le r-1\\d_1,\ldots,d_n\ge 0}}\<\tau^{\alpha_1}_{d_1}\cdots\tau^{\alpha_n}_{d_n}(\tau^{-1}_0)^b\>^{\frac{1}{r},\ext}_h t^{\alpha_1}_{d_1}\cdots t^{\alpha_n}_{d_n}.
$$
\begin{conj}
For any $g\ge 0$,
$$
F^{\frac{1}{r},\ext}_g=(-r)^{\frac{1-g}{2}}\phi_g\left(t^{\le r-2}_*,\frac{1}{\sqrt{-r}}t^{r-1}_*\right).
$$
\end{conj}

%%%%%%%%%%%%%%%%%%%%%%%%%%%%%%%%%%%%%%%%%%%%%%%%%%%%%%%%%%%%
%%%%%%%%%%%%%%%%%%%%%%%%%%%%%%%%%%%%%%%%%%%%%%%%%%%%%%%%%%%%

\section{Open--closed correspondence}\label{sec:o-c corr}

In the forthcoming work~\cite{BCT}, the authors generalize the definition of the moduli space of $r$-spin structures and its Witten class to the setting of genus-zero surfaces with boundary, or {\it disks}.  We summarize the construction of \cite{BCT} in this section and explain the connection between open and closed extended theory.

\subsection{Open $r$-spin theory}

A Riemann surface with boundary is a tuple $(C, \phi, \Sigma, \{z_i\}, \{x_j\})$, where $C$ is an orbifold curve equipped with an involution $\phi: C \rightarrow C$ that realizes $|C|$ as a union of two copies of the Riemann surface with boundary $\Sigma$, glued along their common boundary; we write
\[|C| = \Sigma \cup_{\d \Sigma} \overline{\Sigma}.\]
Here, $z_1, \ldots, z_n \in C$ are the {\it internal marked points} (whose images in $|C|$ lie in $\Sigma \setminus \d\Sigma$, and each of which has a partner $\overline{z}_i := \phi(z_i)$ whose image lies in  $\overline{\Sigma}$) and $x_1, \ldots, x_m \in \d\Sigma$ are the {\it boundary marked points}.  We define a {\it graded $r$-spin structure} on a Riemann surface with boundary as an orbifold line bundle $S$ on $C$ together with an isomorphism
\[S^{\otimes r} \cong \omega_C \otimes \O\left(-\sum_{i=1}^n \alpha_i[z_i] - \sum_{i=1}^n \alpha_i[\overline{z}_i] - \sum_{j=1}^m (r-2)[x_j]\right),\]
an involution $\widetilde{\phi}:S \rightarrow S$ lifting $\phi$, and an additional structure that we refer to as a {\it grading}.  (Roughly, a grading is an involution-invariant section of $S$ on the complement of the special points in $\d\Sigma$ that changes sign at each boundary marked point, this notion was first defined for $r=2$ in \cite{ST0}.)  Here, we assume that the internal twists $\alpha_1, \ldots, \alpha_n$ lie in the range $\{0,1,\ldots, r-1\}$.  Note that we can re-express the genus in terms of the number $b$ of boundary components of $\Sigma$ and the genus $g_c$ of the closed surface obtained from $\Sigma$ by gluing a disk to each boundary component:
\[g = b + g_c -1.\]
For example, a disk has $g = g_c =0$ and $b=1$.

In \cite{BCT} we construct the moduli space $\M^{1/r}_{0,m,\{\alpha_1, \ldots, \alpha_n\}}$ parameterizing graded $r$-spin disks, and prove that it is nonempty exactly when a certain congruence condition on the twists $\alpha_i$ is satisfied.  There should be no difficulty with constructing the all-genus generalization $\M^{1/r}_{g,m,\{\alpha_1, \ldots, \alpha_n\}}$ and this space is nonempty if and only if
\begin{equation}\label{eq:open_rk_constraints}
e_o: = \frac{(g+m-1)(r-2)+2\sum_{i=1}^n\alpha_i}{r}\in \mathbb{N} \;\;\;\; \text{ and } \;\;\;\; e_o\equiv1+m+g \mod 2.
\end{equation}
%though the construction of $\M^{1/r}_{g,m,\{\alpha_1, \ldots, \alpha_n\}}$ currently only exists for $g=0$.

There are line bundles $\mathbb{L}_i$ for each $i=1, \ldots, n$, defined by the cotangent line to the orbifold curve $C$ at the $i$th internal marked point.  Furthermore, in genus zero, there is a real analogue of Witten's bundle,
\[\mathcal{W} := (R^0\pi_*(\mathcal{S}^{\vee} \otimes \omega_{\pi}))_+,\]
a real-rank-$e_o$ bundle over $\M^{1/r}_{0,m,\{\alpha_1, \ldots, \alpha_n\}}$ whose fiber over $(C, \phi,\Sigma,\{z_i\},\{x_j\},S,\widetilde\phi)$ consists of $\widetilde{\phi}$-invariant sections of $S^{\vee} \otimes \omega_C$.

Because $\M^{1/r}_{0,m,\{\alpha_1, \ldots, \alpha_n\}}$ has boundary, one must work with a relative version of the Chern classes of the cotangent line bundles and the Witten bundles.  In \cite{BCT} the requisite canonical boundary conditions are defined over the space $\overline{\mathcal{PM}}^{1/r}_{0,m,\{\alpha_1, \ldots, \alpha_n\}},$ which is some canonical perturbation of the space $\M^{1/r}_{0,m,\{\alpha_1, \ldots, \alpha_n\}}.$ From here, we define open $r$-spin correlators by
\begin{equation}
\label{eq:correlators}
\left\langle\prod_{i=1}^n\tau_{d_i}^{\alpha_i}\;\sigma^m\right\rangle^{\frac{1}{r},o}_{0} :=
\int_{\overline{\mathcal{PM}}^{1/r}_{0,m,\{\alpha_1, \ldots, \alpha_n\}}}e\left(\cW\oplus\bigoplus_{i=1}^n\CL_i^{\oplus d_i},s_{\text{canonical}}\right),
\end{equation}
where $e(E,s)$ is the Euler class relative to the canonical boundary conditions $s$.  Alternatively, the number in \eqref{eq:correlators} can be defined as a weighted, signed count of the number of zeroes of a generic extension of $s_{\text{canonical}}$, defining this count to be zero unless $e_o+2\sum_{i=1}^n d_i = m+2n-3+3g$, or in other words, unless
\begin{equation}\label{eq:when_int_is_number}
e_o+2\sum_{i=1}^n d_i = m+2n - 6 + 3g_c+3b.
\end{equation}
One of the key results of \cite{BCT} is that these intersection numbers are independent of the specific choice of $s_{\text{canonical}}$.

We define a generating function for genus-zero open $r$-spin theory by
\[F_0^{\frac{1}{r},o}(t_*^*, s) := \sum_{\substack{n,m\ge 0\\2n+m-2>0}} \sum_{\substack{ 0 \leq \alpha_1, \ldots, \alpha_n \leq r-1\\ d_1, \ldots, d_n \geq 0}} \frac{1}{n!m!} \left\langle \prod_{i=1}^n \tau_{d_i}^{\alpha_i} \sigma^m \right\rangle_0^{\frac{1}{r},o} t_{d_1}^{\alpha_1} \cdots t_{d_n}^{\alpha_n} s^m.\]

\begin{rmk}
We caution the reader, that, similarly to the closed extended theory, the Ramond vanishing property doesn't hold for the open $r$-spin correlators.
\end{rmk}

\begin{rmk}
In the case where $r=2$ and all of the insertions are Neveu--Schwarz, open $r$-spin theory is equivalent to the intersection theory of disks constructed by Pandharipande, Solomon, and the third author in \cite{PST14}.  This is currently the only case in which we can extend the theory to higher genus \cite{ST1}, calculate all numbers \cite{Tes15}, and prove the relationship to the wave function in all genus \cite{Bur16, BT17}.
\end{rmk}

\subsection{Connection to closed extended theory}

In \cite[Theorem~1.3]{BCT}, we prove that the genus-zero open $r$-spin potential is related to the closed extended potential in the following way:
\begin{gather*}
F^{\frac{1}{r},o}_0(t^0_*,t^1_*,\ldots,t^{r-1}_*,s)=-\frac{1}{r}\left.F^{\frac{1}{r},\ext}_0\right|_{t^{r-1}_d\mapsto t^{r-1}_d-r\delta_{d,0}s}+\frac{1}{r}F^{\frac{1}{r},\ext}_0.
\end{gather*}
We currently do not know of a geometric explanation for the intimate relation between these two theories.  Nevertheless, let us explore more explicitly what is known.

Heuristically, the dictionary between closed extended and open $r$-spin theory is given by
\begin{enumerate}
\item replacing a marked point with twist $-1$ by a boundary component, and
\item replacing a marked point with twist $r-1$ by a boundary marked point.
\end{enumerate}
To make this more precise, we first observe that these exchanges are compatible with the rank-dimension constraints for the two theories.  That is, replacing every boundary component in the open theory with a marked point of twist $-1$ converts equation~\eqref{eq:when_int_is_number} into equation~\eqref{eq:when_int_is_number_closed_extended}, and replacing an internal marked point of twist $r-1$ by a boundary marked point leaves~\eqref{eq:when_int_is_number} invariant.\footnote{These observations are mostly numerological at this point, since the open theory is currently only defined for disks and the closed extended theory in genus zero with a single $-1$ twist.}  Moreover, at the level of intersection numbers, we have the relation
\begin{gather*}
\<\prod_{i=1}^n\tau^{\alpha_i}_{d_i}\sigma^m\>^{\frac{1}{r},o}_0=
\begin{cases}
0,&\text{if $m=0$},\\
(-r)^{m-1}\left\langle\tau_0^{-1}\displaystyle\prod_{i=1}^n\tau^{\alpha_i}_{d_i}(\tau^{r-1}_0)^m\right\rangle^{\frac{1}{r},\ext}_0&\text{if $m\ge 1$},
\end{cases}
\end{gather*}
which realizes the above dictionary when there is a single boundary component.

Moreover, this dictionary matches the topological recursion relations in genus zero.  In \cite{BCT}, we prove two topological recursion relations for open $r$-spin theory.  First, for any $i\in[n]$ with $d_i>0$ and any $j\in([n]\setminus\{i\})$, we have
\begin{align}\label{eq:open_int_trr1}
\<\prod_{l\in [n]}\tau_{d_l}^{\alpha_l}\sigma^m\>^{\frac{1}{r},o}_0=&\sum_{\substack{I\coprod J = [n]\backslash\{i\}\\j\in J}}\sum_{\alpha=-1}^{r-2}\<\tau_0^{\alpha}\tau_{d_i-1}^{\alpha_i}\prod_{l\in I}\tau_{d_l}^{\alpha_l}\>^{\frac{1}{r},\ext}_0\<\tau_0^{r-2-\alpha}\prod_{l\in J}\tau_{d_l}^{\alpha_l}\sigma^m\>^{\frac{1}{r},o}_0+\\
&+\sum_{\substack{I\coprod J = [n]\setminus\{i\}\\m_1+m_2=m\\j\in J}}\frac{m!}{m_1!m_2!}\<\tau_{d_i-1}^{\alpha_i}\prod_{l\in I}\tau_{d_l}^{\alpha_l}\sigma^{m_1}\>^{\frac{1}{r},o}_0\<\sigma\prod_{l\in J}\tau_{d_l}^{\alpha_l}\sigma^{m_2}\>^{\frac{1}{r},o}_0\notag.
\end{align}
Second, if $m\ge 1$, then for any $i\in[n]$ with $d_i>0$, we have
\begin{align}\label{eq:open_int_trr2}
\<\prod_{l\in [n]}\tau_{d_l}^{\alpha_l}\sigma^m\>^{\frac{1}{r},o}_0=&\sum_{I\coprod J = [n]\backslash\{i\}}\sum_{\alpha=-1}^{r-2}\<\tau_0^{\alpha}\tau_{d_i-1}^{\alpha_i}\prod_{l\in I}\tau_{d_l}^{\alpha_l}\>^{\frac{1}{r},\ext}_0\<\tau_0^{r-2-\alpha}\prod_{l\in J}\tau_{d_l}^{\alpha_l}\sigma^m\>^{\frac{1}{r},o}_0+\\
&+\sum_{\substack{I\coprod J = [n]\setminus\{i\}\\m_1+m_2=m-1}}\frac{(m-1)!}{m_1!m_2!}\<\tau_{d_i-1}^{\alpha_i}\prod_{l\in I}\tau_{d_l}^{\alpha_l}\sigma^{m_1}\>^{\frac{1}{r},o}_0\<\sigma\prod_{l\in J}\tau_{d_l}^{\alpha_l}\sigma^{m_2+1}\>^{\frac{1}{r},o}_0\notag.
\end{align}
Under the above dictionary, every term on the right-hand side of~\eqref{eq:open_int_trr1} or~\eqref{eq:open_int_trr2} corresponds to a single term on the right-hand side of~\eqref{eq:NS_TRR} or~\eqref{eq:Ramond_TRR} with $m$ Ramond marked points.

%Note that terms on the right-hand side of equations~\eqref{eq:NS_TRR} or~\eqref{eq:Ramond_TRR} may come from different terms in~\eqref{eq:open_int_trr1} or~\eqref{eq:open_int_trr2}, as the open theory also allows Ramond markings.

Since marked points of twist $r-1$ in closed extended theory may have descendents, one would expect the open-closed correspondence to generalize to that setting.  In \cite{Bur15}, the first author conjectured the precise equations that the open theory for $r=2$ should satisfy if it incorporates boundary descendents.  The construction of boundary descendents when $r=2$ and all insertions are Neveu--Schwarz, was carried by Solomon and the third author, and will appear in the near future. It is also known how to construct these descendents for any $r$ in genus zero, and that will also appear in a forthcoming work.  The resulting topological recursion relations in genus zero are as follows.  First, if $i \in [m]$, $b_i>0$, and $j \in [n]$, then
\begin{gather}\label{eq:open_boundary_point_trr1}
\<\prod_{h\in [m]}\sigma_{b_h}\prod_{l\in [n]}\tau_{d_l}^{\alpha_l}\>^{\frac{1}{r},o}_0=\sum_{\substack{K_I\coprod K_J=[m]\backslash\{i\}\\I\coprod J = [n]\\j\in J}}\<\sigma_{b_i-1}\prod_{h\in K_I}\sigma_{b_h}\prod_{l\in I}\tau_{d_l}^{\alpha_l}\>^{\frac{1}{r},o}_0\<\sigma\prod_{h\in K_J}\sigma_{b_h}\prod_{l\in J}\tau_{d_l}^{\alpha_l}\>^{\frac{1}{r},o}_0,
\end{gather}
where $\sigma_{b}$ corresponds to $b$ descendents at a boundary marked point; and second, if $i\in[m]$, $b_i>0$, and $j\in[m]\backslash\{i\}$, then
\begin{gather}\label{eq:open_boundary_point_trr2}
\<\prod_{h\in [m]}\sigma_{b_h}\prod_{l\in [n]}\tau_{d_l}^{\alpha_l}\>^{\frac{1}{r},o}_0=\sum_{\substack{K_I\coprod K_J=[m]\backslash\{i\}\\I\coprod J = [n]\\j\in K_J}}\<\sigma_{b_i-1}\prod_{h\in K_I}\sigma_{b_h}\prod_{l\in I}\tau_{d_l}^{\alpha_l}\>^{\frac{1}{r},o}_0\<\sigma\prod_{h\in K_J}\sigma_{b_h}\prod_{l\in J}\tau_{d_l}^{\alpha_l}\>^{\frac{1}{r},o}_0.
\end{gather}
These two equations indeed transform to~\eqref{eq:Ramond_TRR} under the open-closed dictionary.

Perhaps the most surprising effect of the open-closed correspondence is the $-1$ TRR (equation~\eqref{eq:-1_TRR}). If we believe the dictionary, then this equation suggests, on the open side, the existence of ``cotangent line classes" corresponding to a bundle $\CL_\text{boun}$ associated to a boundary component. These classes should satisfy the following equation for $h>0$ and $i,j\in[n]$:
\begin{align}\label{eq:bdry_trr}
\<\sigma^\text{boun}_h\prod_{p\in [n]}\tau_{d_p}^{\alpha_p}\prod_{q\in [m]}\sigma_{b_q}\>^{\frac{1}{r},o}_0=&\sum_{\substack{I\coprod J = [n]\\i,j\in J}}\sum_{\alpha=0}^{r-1}\<\sigma^\text{boun}_{h-1}\tau_0^{\alpha}\prod_{p\in I}\tau_{d_p}^{\alpha_q}\prod_{q\in [m]}\sigma_{b_q}\>^{\frac{1}{r},o}_0\<\tau_0^{r-2-\alpha}\prod_{p\in J}\tau_{d_p}^{\alpha_p}\>^{\frac{1}{r},\ext}_0\\
&+\sum_{\substack{I\coprod J = [n]\\K_I\coprod K_J = [m]\\i,j\in J}}\<\sigma^\text{boun}_{h-1}
\sigma\prod_{p\in I}\tau_{d_p}^{\alpha_p}\prod_{q\in K_I}\sigma_{b_q}\>^{\frac{1}{r},o}_0\<\prod_{p\in J}\tau_{d_p}^{\alpha_p}\prod_{q\in K_J}\sigma_{b_q}\>^{\frac{1}{r},o}_0,\notag
\end{align}
where $\sigma^\text{boun}_h$ corresponds to $h$ copies of $\CL_\text{boun}$. (There are analogous equations, also, if one or both of $i,j$ lies in $[m]$.)  Based on this hint, the first and third authors have constructed a ``class" that satisfies equation~\eqref{eq:bdry_trr}---or, more precisely, a line bundle $\CL_\text{boun}$ and boundary conditions for which generic extensions give rise to~\eqref{eq:bdry_trr}. We leave the details of the construction, however, to future work.

\bibliographystyle{plain}

\bibliography{biblio}

\end{document}